\documentclass{jcmlatex}
\def\JCMvol{xx}
\def\JCMno{x}
\def\JCMyear{202x}
\def\JCMreceived{xxx}
\def\JCMrevised{xxx}
\def\JCMaccepted{xxx}

\newcommand{\norm}[1]{\lVert#1\rVert}
\newcommand{\Vnorm}[1]{\left\lVert#1\right\rVert}

\newcommand{\bmath}[1]{\mbox{\boldmath{$#1$}}}

\newcommand{\nn}{\nonumber}
\renewcommand{\vec}[1]{\mathbf{#1}}
\newcommand{\rd}{\mathrm{d}}
\newcommand{\dA}{\mathrm{d}A}
\newcommand{\ds}{\mathrm{d}s}
\usepackage{color}
\usepackage{multirow}
\usepackage{cleveref}
\usepackage[title]{appendix}

\newtheorem{thm}{Theorem}[section]
\newtheorem{prop}{Proposition}[section]

\newtheorem{lem}{Lemma}[section]

\begin{document}
\markboth{W. Bao and Q. Zhao}{An ES-PFEM for Solid-state Dewetting}

\title{An Energy-Stable Parametric Finite Element Method for Simulating Solid-state Dewetting Problems in Three Dimensions}

\author{Weizhu Bao
\thanks{Department of Mathematics, National University of Singapore, Singapore, 119076\\ Email: matbaowz@nus.edu.sg}
\and
 Quan Zhao
\thanks{Department of Mathematics, National University of Singapore, Singapore, 119076\\ Email: quanzhao90@u.nus.edu}}

\maketitle

\begin{abstract}
We propose an accurate and energy-stable parametric finite element method for solving the sharp-interface continuum model of solid-state dewetting in three-dimensional space. The model describes the motion of the film\slash vapor interface with contact line migration and is governed by the surface diffusion equation with proper boundary conditions at the contact line. We present a weak formulation for the problem, in which the contact angle condition is weakly enforced. By using piecewise linear elements in space and backward Euler method in time, we then discretize the formulation to obtain a parametric finite element approximation, where the interface and its contact line are evolved simultaneously. The resulting numerical method is shown to be well-posed and unconditionally energy-stable. Furthermore, the numerical method is generalized to the case of anisotropic surface energies in the Riemannian metric form.  Numerical results are reported to show the convergence and efficiency of the proposed numerical method as well as the anisotropic effects on the morphological evolution of thin films in solid-state dewetting.
\end{abstract}



\section{Introduction}
A thin solid film deposited on the substrate will agglomerate or dewet to form isolated islands due to surface tension\slash capillarity effects when heated to high enough temperatures but well below the thin film's melting point. This process is referred to as the solid-state dewetting (SSD)~\cite{Thompson12} since the thin film remains solid. In recent years, SSD  has been found wide applications in thin film technologies, and it is emerging as a promising route to produce well-controlled patterns of particle arrays used in sensors \cite{Mizsei93}, optical and magnetic devices \cite{Rath07}, and catalyst formations \cite{Randolph07}. A lot of experimental (e.g.,~\cite{Ye10a,Ye11a,Amram12,Rabkin14,Herz216,Naffouti16,Kovalenko17}) and theoretical efforts (e.g.,~\cite{Jiang12,Srolovitz86a,Wang15,Jiang16,Bao17,Bao17b,Dornel06,Wong00,Kim13,Kan05}) have been devoted to SSD not just because of its importance in industrial applications but also the arising scientific questions within the problem.

In general, SSD can be regarded as a type of open surface evolution problem governed by surface diffusion \cite{Mullins57} and moving contact lines \cite{Srolovitz86}. When the thin film moves along the solid substrate, a moving contact
line forms where the three phases (i.e., solid film, vapor and substrate) meet. This brings an additional kinetic feature to this problem. Recently, different mathematical models and simulation methods have been proposed to study SSD, such as sharp-interface models \cite{Srolovitz86, Wong00, Wang15, Jiang19c}, phase-field models \cite{Jiang12, Dziwnik15b, Naffouti17,Huang19b} and other models including the crystalline formulation \cite{Carter95, Zucker13}, discrete chemical potential method \cite{Dornel06} and kinetic Monte Carlo method \cite{Pierre09b}.

\begin{figure}[!htp]
\centering
\includegraphics[width=0.90\textwidth]{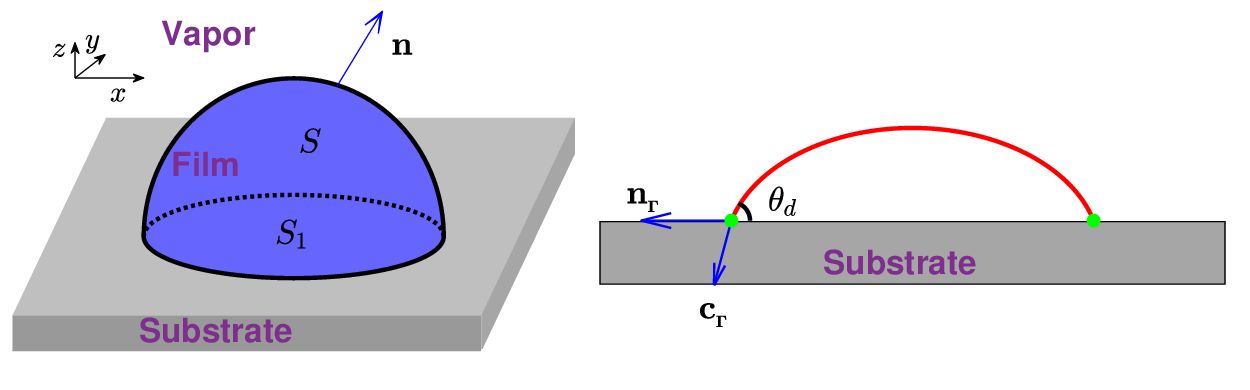}
\caption{Left panel: A geometric setup of SSD, where a thin film (shaded in blue) is deposited on a flat rigid substrate (shaded in gray), $S$ is the film/vapor interface and $S_1$ is the film/substrate interface. Right panel: configuration of the contact angle: $\theta_d=\cos^{-1}(\vec c_{_\Gamma}\cdot\vec n_{_\Gamma})$ at the contact line (green). }
\label{fig:model3d}
\end{figure}

In this work, we will restrict ourselves to the model in \cite{Jiang19c}. It is a sharp-interface model in three dimensions (3D) and was developed based on the thermodynamic variation. As illustrated in Fig.~\ref{fig:model3d}(a), we consider the case when a thin film is deposited on a flat substrate.  The evolving film/vapor interface is described by a moving open surface $S(t)$ with mapping given by (with $\vec X=(x,~y,~z)^T$ or $\vec X=(x_1,~x_2,~x_3)^T$)
 \begin{align}
 \label{eq:Interfacep}
 \vec X(\bmath{\rho},~t)=\Bigl(x_1(\bmath{\rho},t),~x_2(\bmath{\rho},~t),~x_3(\bmath{\rho},t)\Bigr)^T:\,S^0\times[0,~T]\to\mathbb{R}^3,
 \end{align}
 where $S^0=S(0)$ is the initial surface. The film/substrate interface is a flat surface, i.e., a two-dimensional moving domain and denoted by $S_1(t)$. The two interfaces intersect at the contact line and form a closed curve $\Gamma(t):=S(t)\cap S_1(t)$. We assume $\Gamma(t)$ is a simple closed curve and positively orientated with the mapping given by: $ \Gamma(t):=\vec X_{_\Gamma}
(\bmath{\rho},~t),\;\bmath{\rho}\in\Gamma^0=\Gamma(0).$ 

Some relevant geometric parameters are defined as follows: $\vec n$ and $\mathcal{H}$ are the unit outward normal vector and mean curvature of $S(t)$, respectively;  
$\vec c_{_\Gamma}$ and $\vec n_{_\Gamma}$ represent the outward unit conomral vector of $S(t)$ and $S_1(t)$, respectively, and $\nabla_s$ is the surface gradient operator defined in \eqref{eq:sfoperator}. The sharp-interface model of SSD in 3D can be stated as \cite{Jiang19c}:
\begin{subequations}
 \label{eqn:model}
 \begin{align}
 \label{eq:model1}
 &\partial_t\vec X = \Delta_{s}\mathcal{H}\;\vec n,\\
 &\mathcal{H} = -\left(\Delta_{s}{\vec X}\right)\cdot\vec n,
 \label{eq:model2}
 \end{align}
\end{subequations}
where $\Delta_{s}=\nabla_{s}\cdot\nabla_{s}$ is the surface Laplacian operator. The above equations are supplemented with the following conditions at $\Gamma(t)$:

\begin{subequations}\label{eqn:bd}
(i) The contact line condition
 \begin{align}
 \label{eq:bd1}
 x_3(\cdot,~t)|_{_\Gamma} = 0,\qquad t \geq 0.
 \end{align}

(ii) The relaxed contact angle condition
 \begin{align}
 \label{eq:bd2}
 \partial_t\vec X_{_\Gamma} = -\eta \left(\vec c_{_\Gamma}\cdot\vec n_{_\Gamma} - \cos\theta_Y\right)\,\vec n_{_\Gamma},\qquad t \geq 0.
 \end{align}

(iii) The zero-mass flux condition
\begin{align}
\label{eq:bd3}
\left.\left(\vec c_{_\Gamma}\cdot\nabla_{s} \mathcal{H}\right)\right|_{_\Gamma} = 0,\qquad t \geq 0.
\end{align}
\end{subequations}
Here $\theta_Y$ is the Young's equilibrium angle, and $\eta>0$ is the contact line mobility which controls the relaxation rate of the dynamical contact angles to the equilibrium contact angle. Condition (i) ensures that the contact line always move along the substrate surface. When $\eta\to\infty$, condition (ii) collapses to the well-known Young's equation. Condition (iii) implies that there is no mass-flux at the contact line thus the total mass/volume of the thin film is conserved.

The total free energy of the system consists of the film/vapor interfacial energy and the substrate surface energy:
\begin{align}
\label{eq:Totalenergy}
W(t):=|S(t)|-\cos\theta_Y\,|S_1(t)|,
 \end{align}
where $|S(t)|$ and $|S_1(t)|$ denote the surface area of $S(t)$ and $S_1(t)$, respectively. Let $\Omega(t)$ be the region enclosed by $S(t)$ and $S_1(t)$, then the dynamic system obeys the conservation law for the total volume (mass) and the dissipation law for the total surface energy \cite{Jiang19c}
\begin{subequations}
\begin{align}
&\frac{\rd }{\rd t} |\Omega(t)| =\int_{S(t)}\Delta_{s} \mathcal{H} \,\dA \equiv 0,\qquad t \geq 0,\\
&\frac{\rd}{\rd t}W(t) = -\int_{S(t)}\Vnorm{\nabla_{s} \mathcal{H}}^2\;\dA - \eta\int_{\Gamma(t)}\left(\vec c_{_\Gamma}\cdot\vec n_{_\Gamma} - \cos\theta_Y\right)^2\;\ds  \leq 0,
\end{align}
\end{subequations}
where $\norm{\cdot}$ is the Euclidean norm in $\mathbb{R}^3$.

There exist several numerical methods for simulating interface evolution under surface diffusion as well as its applications in SSD. The main difficulty of the problem arises from the complexity of the high-order and nonlinear governing equations and the possible deterioration of the interface mesh during numerical implementation. Therefore, most front tracking methods,
no matter in the framework of finite difference \cite{Wong00, Du10, Wang15,Mayer01} or finite element method \cite{Bansch05,Hausser07, Pozzi08}, generally have to
introduce mesh regularization/smoothing algorithms or artificial tangential velocities to prevent the mesh distortion. By reformulating \eqref{eqn:model}
into a mixed-type formulation as
\begin{subequations}
 \label{eqn:modelnew}
 \begin{align}
 \label{eq:modelnew1}
 &\vec n\cdot \partial_t\vec X = \Delta_{s}\mathcal{H},\\
 &\mathcal{H}\,\vec n = -\Delta_{s}{\vec X},
 \label{eq:modelnew2}
 \end{align}
\end{subequations}
Barrett et al.~\cite{Barrett08JCP, Barrett20} introduced a new variational formulation and designed an elegant semi-implicit parametric finite element method (PFEM) for the surface diffusion equation. The PFEM enjoys a few important and valuable properties including unconditional stability, energy dissipation, and asymptotic mesh equal distribution. It has been successfully extended for solving anisotropic surface diffusion flow under a specific form of convex anisotropies in Riemannian metric form, for simulating the evolution of coupled surface with grain boundary motions and triple junctions~\cite{Barrett08Ani, Barrett10cluster, Barrett10}. Recently, the PFEM has been extended for solving the sharp interface models of SSD in both 2D and 3D \cite{Bao17, Jiang19a, Zhao19b}. However, in those extensions of the PFEM for SSD, they evolve the motions of the interface and the contact lines separately, and do not make full use of the variational structure of the SSD problem. The stability condition depends strongly on the mesh size and the contact line mobility. The convergence rate in space deteriorates and reduces to only first-order.

Motivated by our recent work in 2D~\cite{Zhao20}, the main aim of this work is to propose a new variational formulation and to design an energy-stable parametric finite element method (ES-PFEM) for the 3D SSD problem \eqref{eqn:model} with the boundary conditions \eqref{eq:bd1}-\eqref{eq:bd3}. We first reformulate the relaxed contact angle condition \eqref{eq:bd2} into a Robin-type boundary condition such that it can be naturally absorbed into the weak formulation. This novel treatment helps to maintain the unconditional energy stability of the fully discretized scheme. Furthermore, we extend our ES-PFEM for solving the SSD problem with Riemannian metric type anisotropic surface energies, where the anisotropy is formulated as sums of weighted vector norms \cite{Barrett08Ani}. We report the convergence rate of our ES-PFEM and also investigate the anisotropic effects in SSD via different numerical setups.

The rest of the paper is organized as follows. In section~\ref{sec:weakform}, we present the weak formulation and show that it satisfies the mass conservation and energy dissipation. In section \ref{sec:pfem}, we propose an ES-PFEM as the full discretization of the formulation and show the well-posedness and unconditional energy stability of the numerical method. Subsequently, we extend our numerical method for solving the model of SSD with Riemannian metric type anisotropic surface energies in section \ref{sec:aniso}. Numerical results are reported with convergence test and some applications in section \ref{sec:numr}. Finally, we draw some conclusions in section \ref{sec:con}.

\section{A weak formulation}
\label{sec:weakform}
In this section, we present a weak formulation for the sharp interface model of SSD in
\eqref{eqn:modelnew} (and thus \eqref{eqn:model}) with boundary conditions \eqref{eq:bd1}-\eqref{eq:bd3}, and show the mass conservation and energy dissipation within the weak formulation.

\subsection{The formulation}
\label{ssec:formulation}

We define the function space $L^2(S(t))$ by
\begin{align}
L^2(S(t)):=\bigl\{\psi:S(t)\to\mathbb{R},\quad \int_{S(t)}\psi^2\,\dA<\infty\bigr\},
\end{align}
equipped with the $L^2$-inner product over $S(t)$
\begin{align}
\bigl(u,~v\bigr)_{S(t)}:=\int_{S(t)}u\,v\;\dA,\quad u,v\in L^2(S(t)),
\end{align}
and the associated $L^2$-norm $\Vnorm{u}_{S(t)}:=\sqrt{\left(u,~u\right)_{S(t)}}$. The Sobolev space $H^1(S(t))$ can be naturally defined as
\begin{align}
H^1(S(t)):=\Bigl\{ \psi\in L^2(S(t)),\;{\rm and}\; \underline{D}_i\psi\in L^2(S(t)), i=1,2,3\Bigr\},
\end{align}
where $\underline{D}_if$ is the derivative in weak sense. On the boundary $\Gamma(t)$, we define
\begin{align}
\bigl(u,~v\bigr)_{\Gamma(t)} = \int_{\Gamma(t)}u\,v\,\ds.
\end{align}

The interface velocity of $S(t)$ is given by
\begin{align}
\label{eq:velocity}
\bmath{v}(\vec X(\bmath{\rho},t),~t) = \partial_t\vec X(\bmath{\rho},t),\qquad \forall\vec X:=\vec X(\bmath{\rho},t)\in S(t).
\end{align}
We define the function space for $\bmath{v}$ as
\begin{align}
\mathbb{X}:=H^1(S(t))\times H^1(S(t))\times H^1_0(S(t)),
\end{align}
where the third component of the velocity on $\Gamma(t)$ is zero in time. Multiplying a test function $\psi\in H^1(S(t))$ to \eqref{eq:modelnew1}, integrating over $S(t)$, using integration by parts and the zero-mass flux condition \eqref{eq:bd3},
we obtain
\begin{align}
\bigl(\bmath{v}\cdot\vec n,~\psi\big)_{S(t)} + \bigl(\nabla_{s}\mathcal{H},~\nabla_{s}\psi\bigr)_{S(t)}=0.
\end{align}
Besides, we note the following two equations hold
\begin{subequations}
\label{eqn:im}
\begin{align}
\label{eq:im1}
\vec c_{_\Gamma}\cdot\vec n_{_\Gamma} &= - \frac{1}{\eta}\left(\bmath{v}\cdot\vec n_{_\Gamma}\right)|_{\Gamma(t)} + \cos\theta_Y,\\
\vec c_{_\Gamma} &= \left(\vec c_{_\Gamma}\cdot\vec e_z\right)\vec e_z+(\vec c_{_\Gamma}\cdot\vec n_{_\Gamma})\,\vec n_{_\Gamma},
\label{eq:im2}
\end{align}
\end{subequations}
where \eqref{eq:im1} is a reformulation of the relaxed contact angle condition \eqref{eq:bd2} and \eqref{eq:im2} is a decomposition of $\vec c_{_\Gamma}$ with $\vec e_z=(0,~0,~1)^T$. Now choosing $\bmath{v}=\bmath{g}\in\mathbb{X}$ in \eqref{eq:Kformu}, we obtain 
\begin{align}
0&=\Bigl(\mathcal{H}\,\vec n,~\bmath{g}\Bigr)_{S(t)}-\Bigl(\nabla_s\vec X,~\nabla_s\bmath{g}\Bigr)_{S(t)} + \Bigl(\vec c_{_\Gamma},~\bmath{g}\Bigr)_{\Gamma(t)}\nn\\
&=\Bigl( \mathcal{H}\,\vec n,~\bmath{g}\Bigr)_{S(t)}-\Bigl(\nabla_{s}\vec X,~\nabla_{s}\bmath{g}\Bigr)_{S(t)} + \Bigl(\vec c_{_\Gamma}\cdot\vec n_{_\Gamma},~\bmath{g}\cdot\vec n_{_\Gamma}\Bigr)_{\Gamma(t)}\nn \\
&=\Bigl( \mathcal{H}\,\vec n,~\bmath{g}\Bigr)_{S(t)}-\Bigl(\nabla_{s}\vec X,~\nabla_{s}\bmath{g}\Bigr)_{S(t)} -\frac{1}{\eta}\Bigl(\bmath{v}\cdot\vec n_{_\Gamma},~\bmath{g}\cdot\vec n_{_\Gamma}\Bigr)_{\Gamma(t)} + \cos\theta_Y\Bigl(\bmath{g},~\vec n_{_\Gamma}\Bigr)_{\Gamma(t)},\nn
\end{align}
where in the second equality we have used \eqref{eq:im2} and the fact $\bmath{g}\cdot\vec e_z=0$  on $\Gamma(t)$ and the last equality results from \eqref{eq:im1}.

Collecting these results, we obtain the weak formulation for the sharp-interface model of SSD \eqref{eqn:model} with boundary conditions \eqref{eqn:bd}: Given an initial interface $S(0)$ with boundary $\Gamma(0)$, we use the velocity equation \eqref{eq:velocity} and find the interface velocity $\bmath{v}(\cdot,t)\in \mathbb{X}$ and the mean curvature  $ \mathcal{H}(\cdot,t)\in H^1(S(t))$ such that
\begin{subequations}
\label{eqn:isoweakform}
\begin{align}
\label{eq:isoweakform1}
&\Bigl(\vec n\cdot\bmath{v},~\psi\Bigr)_{S(t)} + \Bigl(\nabla_{s}\mathcal{H},~\nabla_{s}\psi\Bigr)_{S(t)} =0\quad\forall \psi\in H^1(S(t)),\\[0.5em]
&\Bigl(\mathcal{ H}\,\vec n,~\bmath{g}\Bigr)_{S(t)}-\Bigl(\nabla_{s}\vec X,~\nabla_{s}\bmath{g}\Bigr)_{S(t)}-\frac{1}{\eta}\Bigl(\bmath{v}\cdot\vec n_{_\Gamma},~\bmath{g}\cdot\vec n_{_\Gamma}\Bigr)_{\Gamma(t)}+\cos\theta_Y\Bigl(\vec n_{_\Gamma},~\bmath{g}\Bigr)_{\Gamma(t)}\nn\\
&\hspace{4cm} = 0\qquad \forall\bmath{g}\in \mathbb{X}.
\label{eq:isoweakform2}
\end{align}
\end{subequations}
The above weak formulation is an extension of the previous 2D work in Ref.~\cite{Zhao20} to 3D. Similar work for coupled surface or clusters can be found in Refs.~\cite{Barrett10, Barrett10cluster}.

\subsection{Volume\slash mass conservation and energy dissipation}
\label{ssec:massenergy}

For the weak formulation in \eqref{eqn:isoweakform}, we have
\begin{prop}[Mass conservation and energy dissipation]\label{prop:isomassenergy} Let $\left(\vec X,~\bmath{v},~ \mathcal{H}\right)$ be a solution of the weak formulation \eqref{eqn:isoweakform} and \eqref{eq:velocity}, then the total mass of the thin film is conserved, i.e.,
\begin{align}
|\Omega(t)|\equiv|\Omega(0)|,\qquad t \geq 0,
\end{align}
and the total free energy of the system defined in \eqref{eq:Totalenergy} is decreasing, i.e.,
\begin{align}
\label{eq:eedis}
W(t) \leq W(t^\prime) \leq W(0)=|S(0)|-\cos\theta_Y|S_1(0)|,\qquad \forall t \geq t^\prime  \geq 0.
\end{align}
\end{prop}

\begin{proof}
By the Reynolds transport theorem for the moving domain $\Omega(t)$ (see Theorem 33 in \cite{Barrett20}), we have
\begin{align*}
\frac{\rd }{\rd t}|\Omega(t)| = \int_{S(t)}\bmath{v}\cdot\vec n\,\dA,\qquad t\ge0,
\end{align*}
where we have used the fact that the normal velocity of $S_1(t)$ is zero. Now setting $\psi=1$ in \eqref{eq:isoweakform1} yields
\begin{align*}
\frac{\rd }{\rd t}|\Omega(t)| = \Bigl(\nabla_{s} \mathcal{H},~\nabla_{s}1\Bigr)_{S(t)} = 0,\qquad t\ge0,
\end{align*}
which implies the conservation of the total mass.

  Again, using the Reynolds transport theorem for the 2D  moving domain $S_1(t)$ yields
\begin{align*}
\frac{\rd }{\rd t}|S_1(t)| = \int_{\Gamma(t)}\vec n_{_\Gamma}\cdot\bmath{v}\;\ds.
\end{align*}
By noting \eqref{eq:isorey} and \eqref{eq:Totalenergy}, we then have
\begin{align}
\label{eq:weakenergyd1}
\frac{\rd }{\rd t}W(t)=\Bigl(\nabla_{s}\vec X,~\nabla_{s}\bmath{v}\Bigr)_{S(t)} - \cos\theta_Y\Bigl(\vec n_{_\Gamma},~\bmath{v}\Bigr)_{\Gamma(t)}.
\end{align}
Now choosing $\psi= \mathcal{H}$ in \eqref{eq:isoweakform1} and $\bmath{g}=\bmath{v}$ in \eqref{eq:isoweakform2} and using \eqref{eq:weakenergyd1}, we obtain
\begin{align*}
\frac{\rd }{\rd t}W(t)=-\Bigl(\nabla_{s} \mathcal{H},~\nabla_{s}\mathcal{H}\Bigr)_{S(t)} - \frac{1}{\eta}\Bigl(\bmath{v}\cdot\vec n_{_\Gamma},~\bmath{v}\cdot\vec n_{_\Gamma}\Bigr)_{\Gamma(t)} \leq0,
\end{align*}
which immediately implies \eqref{eq:eedis}.
\end{proof}

\section{Parametric finite element approximation}
\label{sec:pfem}

In this section, we present an energy-stable PFEM (ES-PFEM) as the full discretization of the weak formulation \eqref{eqn:isoweakform} by using continuous piecewise linear elements in space and the (semi-implicit) backward Euler method in time. We show the well-posedness and the unconditional energy stability of the resulting method.

\subsection{The discretization}
\label{ssec:discre}

Take $\tau>0$ as the uniform time step size and denote the discrete time levels as $t_m=m\,\tau$ for $m=0,1,\cdots.$  We then approximate the evolution surface $S(t_m)$ by the polygonal surface mesh $S^m$ with a collection of $K$ vertices $\{\vec q_k^m\}_{k=1}^K$ and $N$ mutually disjoint non-degenerate triangles
\begin{align}
\label{eq:polygonalS}
S^m:=\bigcup_{j=1}^N \overline{\sigma_j^m},\quad m \geq 0.
\end{align}
 For $1 \leq j \leq N$, we take $\sigma_j^m:=\triangle\{\vec q_{j_k}^m\}_{k=1}^3$ to indicate that $\{\vec q_{j_1}^m,~\vec q_{j_2}^m,~\vec q_{j_3}^m\}$ are the three vertices of $\sigma_j^m$ and ordered anti-clockwise on the outer surface of $S^m$. The boundary of $S^m$ is a collection of $N_{_\Gamma}$ connected line segments denoted by $\Gamma^m:=\bigcup_{j=1}^{N_{_\Gamma}}\overline{l_j^m}$. Similarly, we take $l_j^m = [\vec p_{j_1}^m,~\vec p_{j_2}^m]$ to indicate that $\vec p_{j_1}^m$ and $\vec p_{j_2}^m$ are the two endpoints of the $j$th line segment and ordered according to the orientation of the curve $\Gamma^m$.

We define the following finite-dimensional spaces over $S^m$ as
\begin{subequations}
\label{eqn:FEMspaces}
\begin{align}
&\mathbb{M}^m:=\left\{\psi\in C(S^m):\psi|_{\sigma_j^m}\in \mathcal{P}^1(\sigma_j^m),\quad\forall1 \leq j \leq N\right\},\\
&\mathbb{M}^m_0:=\left\{\psi\in \mathbb{M}^m: \psi|_{l_j^m} = 0,\quad\forall1 \leq j \leq N_{_\Gamma}\right\},\\
&\mathbb{X}^m:=\,\mathbb{M}^m\times \mathbb{M}^m\times\mathbb{M}_0^m,
\end{align}
\end{subequations}
where $\mathcal{P}^1(\sigma_j^m)$ denotes the spaces of all polynomials with degrees at most $1$ on $\sigma_j^m$.

With the finite element spaces defined above, we can naturally parameterize $S^{m+1}$ over $S^m$ such that $S^{m+1}:=\vec X^{m+1}(\cdot)\in \mathbb{X}^m$. In particular, $S^m:=\vec X^m(\cdot)\in\mathbb{X}^m$ can be considered as the identity function. Let $\vec n^m:=\vec n(\vec X^m)$ be the outward unit normal to $S^m$. It is a piecewise constant vector-valued function and can be defined as
\begin{align}
\label{eq:dnor}
\vec n^m: = \sum_{j=1}^n\vec n_j^m\,\chi_{_{\sigma_j^m}}\quad{\rm with}\quad \vec n_j^m = \frac{(\vec q_{j_2}^m - \vec q_{j_1}^m)\times(\vec q_{j_3}^m - \vec q_{j_1}^m)}{\Vnorm{(\vec q_{j_2}^m - \vec q_{j_1}^m)\times(\vec q_{j_3}^m - \vec q_{j_1}^m)}},
\end{align}
where $\chi_{_E}$ is the usual characteristic function with set $E$, and $\vec n_j^m$ is the outward unit normal on the triangle $\sigma_j^m:=\Delta\{\vec q_{j_k}^m\}_{k=1}^3$. At the boundary $\Gamma^m$, we denote by $\vec n_{_\Gamma}^m:=\vec n_{_\Gamma}(\vec X^m)$ the outward unit conormal vector of $S_1^m$. Then we can compute it at each line segment $l_j^m=[\vec p_{j_1}^m,~\vec p_{j_2}^m]$ as
 \begin{align}
\vec n_{_\Gamma,_j}^m:=\left. \vec n_{_\Gamma}^m\right|_{_{l_j^m}}=\frac{\left(\vec p_{j_2}^m - \vec p_{j_1}^m\right)\times\vec e_z}{\Vnorm{\left(\vec p_{j_2}^m - \vec p_{j_1}^m\right)\times\vec e_z}},\qquad\forall 1 \leq j \leq N_{_\Gamma}.
 \end{align}

If $f,~g$ are two piecewise continuous functions with possible discontinuities across the edges of the triangle element, we define the following mass-lumped inner product to approximate the inner product over $S(t_m)$
\begin{align}
 \bigl( f,~g\bigr)_{S^m}^h := \frac{1}{3}\sum_{j=1}^N \sum_{k=1}^3|\sigma_j^m|\,f\left((\vec q_{j_{_k}}^m)^-\right)\cdot g\left((\vec q_{j_{_k}}^m)^-\right),
 \label{eqn:norm3d}
 \end{align}
where $|\sigma_j^m|=\frac{1}{2}\Vnorm{(\vec q_{j_2}^m - \vec q_{j_1}^m)\times(\vec q_{j_3}^m - \vec q_{j_1}^m)}$ is the surface area of $\sigma_j^m$,  and $f((\vec q_{j_{_k}}^m)^-)$ denotes the one-sided limit of $f(\vec q)$ when $\vec q$ approaches towards $\vec q_{j_{_k}}^m$ from triangle $\sigma_j^m$, i.e., $f((\vec q_{j_{_k}}^m)^-)=\lim\limits_{\sigma_j^m\ni\vec q\rightarrow\vec q_{j_{_k}}^m }f(\vec q)$.

Let $ \mathcal{H}^{m}(\cdot)\in\mathbb{M}^m$ be the numerical solution of the mean curvature at $t_m$. We propose the following ES-PFEM as the full discretization of the weak formulation \eqref{eqn:isoweakform}. Given the polygonal surface $S^0:=\vec X^0(\cdot)\in \mathbb{X}^0$ as a discretization of the initial surface $S(0)$, for $m \geq 0$ we find the evolution surface $S^{m+1}:=\vec X^{m+1}(\cdot)\in\mathbb{X}^m$ and the mean curvature $ \mathcal{H}^{m+1}(\cdot)\in \mathbb{M}^m$ via solving the following two equations
\begin{subequations}
\label{eqn:isopfem}
\begin{align}
\label{eq:isopfem1}
&\Bigl(\frac{\vec X^{m+1} - \vec X^m}{\tau},~\vec n^m\,\psi^h\Bigr)_{S^m}^h + \Bigl(\nabla_{s} \mathcal{H}^{m+1},~\nabla_{s}\psi^h\Bigr)_{S^m} = 0\quad\forall \psi^h \in \mathbb{M}^m,\\[0.5em]
\label{eq:isopfem2}
&\Bigl( \mathcal{H}^{m+1}\,\vec n^m,~\bmath{g}^h\Bigr)_{S^m}^h-\Bigl(\nabla_{s}\vec X^{m+1},~\nabla_{s}\bmath{g}^h\Bigr)_{S^m}+\cos\theta_Y\,\Bigl(\vec n_{_\Gamma}^{m+\frac{1}{2}},~\bmath{g}^h\Bigr)_{\Gamma^m}\\
&\hspace{2cm}-\frac{1}{\eta\,\tau}\Bigl((\vec X_{_\Gamma}^{m+1} - \vec X_{_\Gamma}^m)\cdot\vec n_{_\Gamma}^m,~~\bmath{g}^h\cdot\vec n_{_\Gamma}^m\Bigr)_{\Gamma^m} = 0\quad \forall\bmath{g}^h\in \mathbb{X}^m,\nn
\end{align}
\end{subequations}
where $\nabla_{s}$ is defined over $S^m$ and 
\begin{align}
\label{eq:nsemi}
\vec n_{_{\Gamma}}^{m+\frac{1}{2}}=\frac{1}{2}\left(\partial_s\vec X_{_\Gamma}^m + \partial_s\vec X_{_\Gamma}^{m+1}\right)\times\vec e_z,
\end{align}
with $s$ being the arc length parameter of $\Gamma^m$. We will show in Lemma.~\ref{lem:energy1} that this semi-implicit approximation helps to maintain the energy stability for the substrate energy. 


The proposed numerical method is semi-implicit, i.e. only a linear system needs to be solved at each time step. It conserves the total volume very well up to the truncation error from the time discretization. Besides, a spatial discretization of \eqref{eq:isoweakform1} by using the piecewise linear elements introduces an implicit tangential velocity of the mesh points and result in the good mesh quality. The detailed investigation of this property has been presented in~\cite{Barrett08JCP}.

\subsection{Well-posedness}
\label{ssec:wp}

Let $\mathcal{J}_k^m:=\{1 \leq j \leq N\ |\ \vec q_k^m\in\overline{\sigma_j^m}\}$ be the index collection of triangles that contain the vertex $\vec q_k^m$, $1 \leq k \leq K$. We define the weighted unit normal at the vertex $\vec q_k^m$ as
\begin{align}
\label{eq:weightN}
\bmath{\omega}^m_k:=\frac{1}{\sum_{j\in\mathcal{J}_k^m}|\sigma_j^m|}\left(\sum_{j\in\mathcal{J}_k^m}\,|\sigma_j^m|\,\vec n_j^m\right),
\end{align}
where $\vec n_j^m$ is  defined in \eqref{eq:dnor} as the unit normal of the triangle $\sigma_j^m$.

For the discretization in \eqref{eqn:isopfem}, we have
\begin{thm}[Existence and uniqueness]\label{th:wellp}
Assume that $S^m$ satisfies:
\begin{itemize}
\item[(i)] $\min_{1 \leq j \leq N}|\sigma_j^m|>0$;
\item [(ii)] ${\rm dim}\{\vec n_{_\Gamma,_j}^m\}_{j=1}^{N_{_\Gamma}}=2$;
\item [(iii)] there exist a vertex $\vec q_{k_0}$ on the polygonal curve $\Gamma^m$ such that $\bmath{\omega}_{k_0}^m=(w^m_{k_0,1},~w_{k_0,2}^m,~w_{k_0,3}^m)^T$ satisfies $\left(w_{k_0,1}^m\right)^2+\left(w_{k_0,2}^m\right)^2>0$.
\end{itemize}
If $\cos\theta_Y=0$, i.e., $\theta_Y=\frac{\pi}{2}$, then the linear system in \eqref{eqn:isopfem} admits a unique solution.
\end{thm}

\begin{proof}
It is sufficient to prove the corresponding homogeneous linear system only has zero solution. Therefore we consider the following homogeneous linear system: Find $\Bigl(\vec X^h,~ \mathcal{H}^h\Bigr)\in\Bigl(\mathbb{X}^m,~\mathbb{M}^m\Bigr)$ such that $\forall\psi^h\in \mathbb{M}^m,\;\bmath{g}^h\in\mathbb{X}^h$
\begin{subequations}
\label{eqn:homopfem}
\begin{align}
\label{eq:homopfem1}
&\Bigl(\vec X^{h},~\vec n^m\,\psi^h\Bigr)_{S^m}^h + \Bigl(\nabla_{s} \mathcal{H}^{h},~\nabla_{s}\psi^h\Bigr)_{S^m} = 0,\\[0.5em]
\label{eq:homopfem2}
&\Bigl( \mathcal{H}^{h}\,\vec n^m,~\bmath{g}^h\Bigr)_{S^m}^h-\left(\nabla_{s}\vec X^{h},~\nabla_{s}\bmath{g}^h\right)_{S^m}-\frac{1}{\eta\,\tau}\Bigl(\vec X_{_\Gamma}^{h} \cdot\vec n_{_\Gamma}^m,~~\bmath{g}^h\cdot\vec n_{_\Gamma}^m\Bigr)_{\Gamma^m} = 0.
\end{align}
\end{subequations}
Setting $\psi^h =  \mathcal{H}^{h}$ in \eqref{eq:homopfem1} and $\bmath{g}^h=\vec X^{h}$ in \eqref{eq:homopfem2}, and combining these two equations, we arrive at
\begin{align}
&\Bigl(\nabla_{s}\mathcal{H}^{h},~\nabla_{s} \mathcal{H}^{h}\Bigr)_{S^m} + \Bigl(\nabla_{s}\vec X^{h},~\nabla_{s}\vec X^{h}\Bigr)_{S^m} + \frac{1}{\eta\tau}\Bigl(\vec X_{_\Gamma}^{h}\cdot\vec n^m_{_\Gamma},~\vec X_{_\Gamma}^{h}\cdot\vec n^m_{_\Gamma}\Bigr)_{\Gamma^m} =0.\nn
\end{align}
This implies 
\begin{align*}
\mathcal{H}^h\equiv \mathcal{H}^c,\qquad \vec X^h\equiv \vec X^c,\qquad \vec X_{_\Gamma}^h\cdot\vec n_{_\Gamma}^m=0,
\end{align*}
where $\mathcal{H}^c$ and $\vec X^c$ represent constants. By assumption (ii), we obtain $\vec X_{_\Gamma}^h=\vec 0$, and thus we get $\vec X^h \equiv \vec 0$. Plugging $ \mathcal{H}^c$ and $\vec X^c=\vec 0$ into \eqref{eq:homopfem2} yields
\begin{align}
\label{eq:wp2}
 \mathcal{H}^c\,\Bigl(\bmath{g}^h,~\vec n^m\Bigr)_{S^m}^h = 0,\qquad\forall\bmath{g}^h\in\mathbb{X}^m.
\end{align}

Now we choose $\bmath{g}^h\in\mathbb{X}^m$ such that
\begin{equation}\label{omgh11}
\left.\bmath{g}^h\right|_{\vec q_k^m}=\left\{\begin{array}{ll}
\Bigl(w_{k_0,1}^m, ~w_{k_0,2}^m,~0\Bigr)^T, &k=k_0,\\[0.4em]
\vec 0 , &{\rm otherwise}.
\end{array}\right.\nn
\end{equation}
Plugging $\bmath{g}^h$ into \eqref{eq:wp2} and by noting the mass-lumped norm in \eqref{eqn:norm3d}, we obtain
\begin{align*}
 \mathcal{H}^c\Bigl(\bmath{g}^h,~\vec n^m\Bigr)_{S^m}^h = \frac{ \mathcal{H}^c}{3}\sum_{j=1}^N\sum_{i=1}^3\bmath{g}^h(\vec q_{j_i}^m)|\sigma_j^m|\vec n_j^m = \frac{ \mathcal{H}^c}{3}\sum_{j\in \mathcal{J}_{k_0}^m}|\sigma_j^m|\Bigl[\left(w_{k_0,1}^m\right)^2+\left(w_{k_0,2}^m\right)^2\Bigr]=0.
\end{align*}
By noting the two assumptions (i) and (iii), we obtain $ \mathcal{H}^c=0$. This shows the homogenous system \eqref{eqn:homopfem} has only zero solution. Thus the numerical scheme \eqref{eqn:isopfem} admits a unique solution.
\end{proof}

Assumption (i) implies that each triangle element is non-degenerate, assumption (ii) implies that there exist at least two line segments on the polygonal curve of contact line not parallel to each other, and assumption (iii) implies that the weighted normal vector at $\vec q_{k_0}^m\in\Gamma^m$ is not perpendicular to the substrate surface ($xoy$-plane). We note the well-posedness is only proved when $\cos\theta_Y=0$. By using matrix perturbation theory, we can also show that \eqref{eq:isopfem1}-\eqref{eq:isopfem2} is well-posed as long as $|\cos\theta_Y|\ll 1$. In practical computation, we observe the linear system in \eqref{eqn:isopfem} is always invertible.  


\subsection{Energy dissipation}
\label{ssec:energyd}

For $\vec n_{_\Gamma}^{m+\frac{1}{2}}$ defined in \eqref{eq:nsemi}, we have the following lemma.
\begin{lem}
\label{lem:energy1}
It holds that
\begin{align}
\label{eq:subenergy}
|S_1^{m+1}| - |S_1^m| = \Bigl(\vec n_{_\Gamma}^{m+\frac{1}{2}},~[\vec X_{_\Gamma}^{m+1} - \vec X_{_\Gamma}^m]\Bigr)_{\Gamma^m},\qquad m \geq 0,
\end{align}
where $|S_1^m|$ denotes the surface area enclosed by the closed plane curve $\Gamma^m$ on the substrate.
\end{lem}
\begin{proof}
Denote $\Gamma^m:=\vec X_{_\Gamma}^m(s)$ and $\Gamma^{m+1}:=\vec X_{_\Gamma}^{m+1}(s)$, where $s\in[0,~L^m]$ is the arc length parameter of $\Gamma^m$. We can compute
\begin{align}
\Bigl(\vec n_{_\Gamma}^{m+\frac{1}{2}},~[\vec X_{_\Gamma}^{m+1} - \vec X_{_\Gamma}^m]\Bigr)_{\Gamma^m}&=\underbrace{\frac{1}{2}\,\int_0^{L^m}\left(\partial_s\vec X_{_\Gamma}^{m+1}\times\vec e_z\right)\cdot\vec X_{_\Gamma}^{m+1}\,\ds}_{\mathcal{A}} - \underbrace{\frac{1}{2}\int_0^{L^m}\left(\partial_s\vec X_{_\Gamma}^{m}\times\vec e_z\right)\cdot\vec X_{_\Gamma}^{m}\,\ds}_{\mathcal{B}}\nn \\
&\hspace{0.5em}+\; \underbrace{\frac{1}{2}\int_0^{L^m}\left(\partial_s\vec X_{_\Gamma}^{m}\times\vec e_z\right)\cdot\vec X_{_\Gamma}^{m+1}\,\ds}_{\mathcal{C}}  - \underbrace{\frac{1}{2}\int_0^{L^m}\left(\partial_s\vec X_{_\Gamma}^{m+1}\times\vec e_z\right)\cdot\vec X_{_\Gamma}^{m}\,\ds}_{\mathcal{D}}. \nn
\end{align}
Now we can recast the first two terms as
\begin{align}\label{eq:subenergy1}
\mathcal{A}=\frac{1}{2}\Bigl(\vec n_{_\Gamma}^{m+1},~\vec X_{_\Gamma}^{m+1}\Bigr)_{\Gamma^{m+1}} = |S_1^{m+1}|,\qquad
\mathcal{B}=\frac{1}{2}\Bigl(\vec n_{_\Gamma}^{m},~\vec X_{_\Gamma}^{m+1}\Bigr)_{\Gamma^{m}} = |S_1^{m}|.
\end{align}
For the third term, using integration by parts and the identity $(\vec a\times\vec b)\cdot\vec c = -(\vec c\times\vec b)\cdot\vec a$ yields
\begin{align}
\label{eq:subenergy2}
\mathcal{C}=-\frac{1}{2}\int_0^{L^m}\left(\vec X_{_\Gamma}^{m}\times\vec e_z\right)\cdot\partial_s\vec X_{_\Gamma}^{m+1}\,\ds = \mathcal{D}.
\end{align}
Collecting these results in \eqref{eq:subenergy1} and \eqref{eq:subenergy2}, we obtain \eqref{eq:subenergy}.
\end{proof}

Denote
\begin{align}
\label{eq:Energyd}
W^m: = |S^m| - \cos\theta_Y\,|S_1^m|.
\end{align}
Similar to the previous work in \cite{Barrett08JCP}, we can prove:
\begin{thm}\label{th:EnergyS} Let $\Bigl(\vec X^{m+1},~ \mathcal{H}^{m+1}\Bigr)$ be the numerical solution of \eqref{eqn:isopfem}, then the energy is decreasing during the evolution: i.e.,
\begin{align}
W^{m+1}  \leq W^m \leq W^0= |S^0| -\cos\theta_Y|S_1^0|,\quad m \geq 0.
\label{eq:EnergyS}
\end{align}
Moreover, we have,
\begin{align}
&\sum_{l=1}^{m+1}\Vnorm{\nabla_{s}\mathcal{H}^{l}}^2_{S^{l-1}} + \frac{1}{\eta}\sum_{l=1}^{m+1}\Vnorm{\left(\frac{\vec X^{l}_{_\Gamma} - \vec X^{l-1}_{_\Gamma}}{\tau}\right)\cdot\vec n^{l-1}_{_\Gamma}}^2_{\Gamma^{l-1}}  \leq \frac{W^0-W^{m+1}}{\tau},\quad \forall m  \geq 0,
\label{eq:EnergySW}
\end{align}
where $\norm{\cdot}_{S^l}$ and $\norm{\cdot}_{\Gamma^l}$ are the $L^2$-norm over $S^l$ and $\Gamma^l$, respectively.
\end{thm}
\begin{proof}
Setting $\psi^h = \tau\, \mathcal{H}^{m+1}$ in \eqref{eq:isopfem1} and $\bmath{g}^h = \vec X^{m+1} - \vec X^m$ in \eqref{eq:isopfem2}, combining these two equations yields
\begin{align}
&\tau\,\Vnorm{\nabla_s\mathcal{H}^{+1}}_{S^m}^2 + \Bigl(\nabla_{s}\vec X^{m+1},~\nabla_{s}\left[\vec X^{m+1}-\vec X^m\right]\Bigr)_{S^m}+\frac{1}{\eta\,\tau}\Vnorm{(\vec X_{_\Gamma}^{m+1} - \vec X_{_\Gamma}^m)\cdot\vec n_{_\Gamma}^m}_{\Gamma^m}^2\nn
\\
&\hspace{2cm} -\cos\theta_Y\Bigl(\vec n_{_\Gamma}^{m+\frac{1}{2}},~[\vec X_{_\Gamma}^{m+1} - \vec X_{_\Gamma}^m]\Bigr)_{\Gamma^m} =0.
\label{eq:energy1}
\end{align}
Using the inequality $a(a-b)\geq\frac{1}{2}\left(a^2-b^2\right)$, we get
\begin{align}
\Bigl(\nabla_{s}\vec X^{m+1},~\nabla_{s}\left[\vec X^{m+1}-\vec X^m\right]\Bigr)_{S^m} \geq \frac{1}{2}\Vnorm{\nabla_{s}\vec X^{m+1}}_{S^m}^2-\frac{1}{2}\Vnorm{\nabla_{s}\vec X^m}_{S^m}^2.\nn
\end{align}
By noting \eqref{eqn:isoieq} and choosing $\vec Y=\vec X^{m+1}$, we have 
\begin{align}
\frac{1}{2}\Vnorm{\nabla_{s}\vec X^{m+1}}_{S^m}^2 \geq |S^{m+1}|,\qquad \frac{1}{2}\Vnorm{\nabla_{s}\vec X^m}_{S^m}^2=|S^m|.\nn
\end{align}
This gives 
\begin{align}
\Bigl(\nabla_{s}\vec X^{m+1},~\nabla_{s}\left[\vec X^{m+1}-\vec X^m\right]\Bigr)_{S^m}\geq |S^{m+1}| - |S^m|,
\label{eq:energy2}
\end{align}
Plugging \eqref{eq:energy2} and \eqref{eq:subenergy} into \eqref{eq:energy1} and noting \eqref{eq:Energyd}, we then obtain
\begin{align}
\label{eq:energy4}
W^{m+1} + \tau\Vnorm{\nabla_{s}~ \mathcal{H}^{m+1}}_{S^m}^2 + \frac{1}{\eta\tau}\Vnorm{(\vec X_{_\Gamma}^{m+1} - \vec X_{_\Gamma}^m)\cdot\vec n_{_\Gamma}^m}^2_{\Gamma^m} \leq W^m,
\end{align}
which immediately implies \eqref{eq:EnergyS}.

Reformulating \eqref{eq:energy4} as
\begin{align*}
\Vnorm{\nabla_{s}\mathcal{H}^{l}}_{S^{l-1}}^2 + \frac{1}{\eta}\Vnorm{\left(\frac{\vec X_{_\Gamma}^{l} - \vec X_{_\Gamma}^{l-1}}{\tau}\right)\cdot\vec n_{_\Gamma}^{l-1}}^2_{\Gamma^{l-1}} \leq \frac{W^{l-1}-W^l}{\tau},\qquad l \geq 1,
\end{align*}
summing up for $l=1,2,\cdots,m+1$, we obtain \eqref{eq:EnergySW}.
\end{proof}


\section{For  anisotropic surface energies in Riemannian metric form}
\label{sec:aniso}

In this section, we first present the sharp-interface model of SSD with anisotropic surface energies in the Riemannian metric form, and then generalize our ES-PFEM for solving the anisotropic model.

\subsection{The model and its weak form}
\label{ssec:aniweakform}

In the case of anisotropic surface energies, the total free energy of the SSD system reads
\begin{align}
\label{eq:anitotalW}
W_\gamma(t):=\int_{S(t)}\gamma(\vec n)\,\dA - \cos\theta_Y |S_1(t)|,
\end{align}
where $\gamma(\vec n)$ represents the anisotropic surface energy density. In the current work we restrict ourselves to the anisotropy which is given by the sums of weighted vector norms as discussed in \cite{Barrett08Ani}:
\begin{align}
\label{eq:aniso}
\gamma(\vec n) = \sum_{i=1}^L\gamma_i(\vec n)=\sum_{i=1}^L\sqrt{\vec n^T\, G_i\,\vec n},
\end{align}
where $G_i\in\mathbb{R}^{3\times3}$ is symmetric positive definite for $i = 1,\cdots, L$. Some typical examples of $\gamma(\vec n)$ are: (1) isotropic surface energy: $L=1$, $G_1=\mathbb{I}$, which gives $\gamma(\vec n)\equiv 1$; (2) ellipsoidal surface energy: $L=1$,
$G_1={\rm diag}(a_1^2,  a_2^2, a_3^2)$, which gives the ellipsoidal surface energy
\begin{align}
\label{eq:ellip}
\gamma(\vec n) = \sqrt{\sum_{i=1}^3 a_i^2n_i^2},\qquad a_i> 0;
\end{align}
and (3) ``cusped" surface energy: $L=3$,
$G_1 ={\rm diag}(1, \delta^2, \delta^2)$,
 $G_2 ={\rm diag}(
\delta^2, 1, \delta^2)$,
$G_3 ={\rm diag}(
\delta^2, \delta^2, 1)$,
which gives the ``cusped'' surface energy
\begin{align}
\label{eq:cuspgamma}
\gamma(\vec n) = \sum_{i=1}^3\sqrt{(1-\delta^2) n_i^2 + \delta^2\Vnorm{\vec n}^2}.
\end{align}
In fact, \eqref{eq:cuspgamma} is a smooth regularization of $\gamma(\vec n)=\sum_{i=1}^3|n_i|$ by introducing a small parameter $0<\delta\ll1$. For other choices of $L$ and $G_i$, readers can refer to Ref.~\cite{Barrett08Ani} and references therein.

 With the mapping $\vec X(\cdot, t)$ given in \eqref{eq:Interfacep}, the sharp interface model of SSD with anisotropic surface energies can be derived as \cite{Jiang19c}
\begin{subequations}
 \label{eqn:animodel}
 \begin{align}
 \label{eq:animodel1}
 &\partial_t\vec X = \Delta_{s}\mu\;\vec n,\quad t \geq 0,\\
 &\mu = \nabla_{s}\cdot\bmath{\xi},\quad \bmath{\xi}(\vec n)=\sum_{i=1}^L\frac{1}{\gamma_i(\vec n)}G_i\,\vec n,
 \label{eq:animodel2}
 \end{align}
\end{subequations}
where $\mu(\vec X,~t)$ is the chemical potential and $\bmath{\xi}(\vec n)$ is the associated Cahn-Hoffman $\bmath{\xi}$-vector \cite{Hoffman72,Cahn74}. The above equations are supplemented with the contact line condition in
\eqref{eq:bd1} and the following two additional boundary conditions at the contact line $\Gamma(t)$:

(ii') The relaxed anisotropic contact angle condition
 \begin{align}
 \label{eq:anibd2}
 \partial_t\vec X_{_\Gamma} = -\eta \left(\vec c_{_\Gamma}^\gamma\cdot\vec n_{_\Gamma} - \cos\theta_Y\right)\,\vec n_{_\Gamma},\qquad t \geq 0.
 \end{align}

(iii') The anisotropic zero-mass flux condition
\begin{align}
\label{eq:anibd3}
\left(\vec c_{_\Gamma}\cdot\nabla_{s}\mu\right)|_{_\Gamma} = 0,\qquad t \geq 0;
\end{align}
where $\vec c_{_\Gamma}^\gamma$ is the anisotropic conormal vector defined as
\begin{align}
\label{eq:anicon}
\vec c_{_\Gamma}^\gamma=(\bmath{\xi}\cdot\vec n)\,\vec c_{_\Gamma} -(\bmath{\xi}\cdot\vec c_{_\Gamma})\,\vec n.
\end{align}
When $\eta\to\infty$, \eqref{eq:anibd2} will reduce to the anisotropic Young's equation: $\vec c_{_\Gamma}^\gamma\cdot\vec n_{_\Gamma} - \cos\theta_Y=0$.

To obtain the weak formulation for the above model, we shall make use of the anisotropic surface gradient defined in Appendix \ref{sec:app1}. In a similar manner as we did in the isotropic case, we choose $\bmath{v}=\bmath{g}\in\mathbb{X}$ in \eqref{eq:muform2}, use the decomposition $\vec c_{_\Gamma}^\gamma = \left(\vec c_{_\Gamma}^\gamma\cdot\vec e_z\right)\,\vec e_z + \left(\vec c_{_\Gamma}^\gamma\cdot\vec n_{_\Gamma}\right)\,\vec n_{_\Gamma}$ and the relaxed anisotropic contact angle condition in \eqref{eq:anibd2}. This gives
\begin{align*}
0&=\Bigl(\mu\,\vec n,~\bmath{g}\Bigr)_{S(t)} -\sum_{i=1}^L\left(\gamma_i(\vec n),~\left(\nabla_{s}^{\widetilde{G}_i}\,\vec X,~\nabla_{s}^{\widetilde{G}_i}\bmath{g}\right)_{\widetilde{G}_i}\right)_{S(t)} + \Bigl(\vec c_{_\Gamma}^\gamma\cdot\vec n_{_\Gamma},~\bmath{g}\cdot\vec n_{_\Gamma}\Bigr)_{\Gamma(t)}\\
&=\Bigl(\mu\,\vec n,~\bmath{g}\Bigr)_{S(t)} -\sum_{i=1}^L\left(\gamma_i(\vec n),~\left(\nabla_{s}^{\widetilde{G}_i}\,\vec X,~\nabla_{s}^{\widetilde{G}_i}\bmath{g}\right)_{\widetilde{G}_i}\right)_{S(t)}-\frac{1}{\eta}\Bigl(\partial_t\vec X_{_\Gamma}\cdot\vec n_{_\Gamma},~\bmath{g}\cdot\vec n_{_\Gamma}\Bigr)_{\Gamma(t)}\nn\\
&\hspace{1cm}+\cos\theta_Y\Bigl(\vec n_{_\Gamma},~\bmath{g}\Bigr)_{\Gamma(t)}.
\end{align*}

Collecting these results, the generalized weak formulation for anisotropic case is given as follows: for $t>0$ we use the velocity equation \eqref{eq:velocity} and seek the interface velocity $\bmath{v}(\cdot, ~t)\in \mathbb{X}$ as well as the chemical potential $\mu(\cdot,~t)\in H^1(S(t))$ via solving the following two equations
\begin{subequations}
\label{eqn:aniweakform}
\begin{align}
\label{eq:aniweakform1}
&\Bigl(\bmath{v}\cdot\vec n,~\psi\Bigr)_{S(t)} + \Bigl(\nabla_{s}\mu,~\nabla_{s}\psi\Bigr)_{S(t)} =0\quad\forall \psi\in H^1(S(t)),\\
\label{eq:aniweakform2}
&\Bigl(\mu\,\vec n,~\bmath{g}\Bigr)_{S(t)}-\sum_{i=1}^L\left(\gamma_i(\vec n),~\left(\nabla_{s}^{\widetilde{G}_i}\vec X,~\nabla_{s}^{\widetilde{G}_i}\bmath{g}\right)_{\widetilde{G}_i}\right)_{S(t)}-\frac{1}{\eta}\Bigl(\bmath{v}\cdot\vec n_{_\Gamma},~\bmath{g}\cdot\vec n_{_\Gamma}\Bigr)_{\Gamma(t)}\\
&\hspace{2.5cm}+\cos\theta_Y\Bigl(\vec n_{_\Gamma},~\bmath{g}\Bigr)_{\Gamma(t)} = 0\quad \forall\bmath{g}\in \mathbb{X}.\nn
\end{align}
\end{subequations}

For the above weak formulation \eqref{eqn:aniweakform}, we have
\begin{prop}[Mass conservation and energy dissipation] Let $\left(\vec X,~\bmath{v},~ \mathcal{H}\right)$ be a solution of the weak formulation \eqref{eqn:aniweakform} and \eqref{eq:velocity}, then the total mass of the thin film is conserved, i.e.,
\begin{align}
|\Omega(t)|\equiv|\Omega(0)|,\qquad t \geq 0,
\end{align}
and the total free energy of the system defined in \eqref{eq:anitotalW} is decreasing, i.e.,
\begin{align}
\label{eq:eedis2}
W_\gamma(t)\leq W_\gamma(t^\prime) \leq W_\gamma(0)=\sum_{i=1}^L\int_{S(0)}\gamma_i(\vec n(\vec X,0))\,\rd A-\cos\theta_Y|S_1(0)|,\qquad \forall t \geq t^\prime  \geq 0.
\end{align}
\end{prop}
\begin{proof}
The proof of the mass conservation is the same as the that in Proposition \ref{prop:isomassenergy}. For the energy dissipation, by noting \eqref{eq:anirey} and \eqref{eq:anitotalW}, we have
\begin{align}
\label{eq:aniim}
\frac{\rd }{\rd t}W_\gamma(t)= \sum_{i=1}^L\int_{S(t)}\gamma_i(\vec n)\left(\nabla_{s}^{\widetilde{G}_i}\vec X,~\nabla_{s}^{\widetilde{G}_i}\bmath{v}\right)_{\widetilde{G}_i}\,\rd A - \cos\theta_Y\Bigl(\vec n_{_\Gamma},~\bmath{v}\Bigr)_{\Gamma(t)}.
\end{align}
Setting $\psi=\mu$ in \eqref{eq:aniweakform1} and $\bmath{g}=\bmath{v}$ in \eqref{eq:aniweakform2}, and using \eqref{eq:aniim}, we obtain
\begin{align}
\frac{\rd }{\rd t}W_\gamma(t)=-\Big(\nabla_s\mu,~\nabla_s\mu\Bigr)_{S(t)}-\frac{1}{\eta}\Bigl(\bmath{v}\cdot\vec n_{_\Gamma},~\bmath{v}\cdot\vec n_{_\Gamma}\Bigr)_{\Gamma(t)}\leq 0,
\end{align}
which implies the energy dissipation.
\end{proof}

\subsection{The ES-PFEM and its properties}
\label{ssec:anidis}

Following the discretization in \cref{ssec:discre}, we propose an ES-PFEM as the full discretization of the weak formulation \eqref{eqn:aniweakform} as follows: Given the polygonal surface $S^0:=\vec X^0(\cdot)\in \mathbb{X}^0$, for $m \geq 0$ we find $S^{m+1}:=\vec X^{m+1}(\cdot)\in\mathbb{X}^m$ and the chemical potential $\mu^{m+1}(\cdot)\in \mathbb{M}^m$ via solving the following two equations
\begin{subequations}
\label{eqn:anipfem}
\begin{align}
\label{eq:anipfem1}
&\Bigl(\frac{\vec X^{m+1} - \vec X^m}{\tau},~\vec n^m\,\psi^h\Bigr)_{S^m}^h + \Bigl(\nabla_{s}\mu^{m+1},~\nabla_{s}\psi^h\Bigr)_{S^m} = 0\quad\forall \psi^h \in \mathbb{M}^m,\\[0.6em]
&\Bigl(\mu^{m+1}\,\vec n^m,~\bmath{g}^h\Bigr)_{S^m}^h-\sum_{i=1}^L\left(\gamma_i(\vec n^m),~\left(\nabla_{s}^{\widetilde{G}_i}\,\vec X^{m+1},~\nabla_{s}^{\widetilde{G}_i}\bmath{g}^h\right)_{\widetilde{G}_i}\right)_{S^m}+\cos\theta_Y\,\Bigl(\vec n_{_\Gamma}^{m+\frac{1}{2}},~\bmath{g}^h\Bigr)_{\Gamma^m}\nn\\
&\hspace{2cm}-\frac{1}{\eta\,\tau}\Bigl([\vec X_{_\Gamma}^{m+1} - \vec X_{_\Gamma}^m]\cdot\vec n_{_\Gamma}^m,~~\bmath{g}^h\cdot\vec n_{_\Gamma}^m\Bigr)_{\Gamma^m} = 0\quad \forall\bmath{g}^h\in \mathbb{X}^m,\label{eq:anipfem2}
\end{align}
\end{subequations}
where $\nabla_{s}$ and $\nabla_{s}^{\widetilde{G}_i}$ are defined on $S^m$, and $\vec n_{\Gamma}^{m+\frac{1}{2}}$ is defined in \eqref{eq:nsemi}. In practical computation,  $\nabla_{s}^{\widetilde{G}}f$ on a typical triangle $\sigma=\Delta\{\vec q_k\}_{k=1}^3$ is computed by using the definition \eqref{eq:anigradient} as
\begin{align*}
\nabla_{s}^{\widetilde{G}}f|_{\sigma}:= (\nabla_{s}f\cdot\vec t_1)\,\vec t_1 + (\nabla_{s}f\cdot\vec t_2)\,\vec t_2,
\end{align*}
where $\vec t_1$ and $\vec t_2$ are parallel to the triangle $\sigma$ and satisfy \eqref{eq:ortho}. In the case when $L=1$ and $G_1=\mathbb{I}$, the numerical scheme \eqref{eqn:anipfem} reduces to the scheme  \eqref{eqn:isopfem}.

For the discretization in \eqref{eqn:anipfem}, similar to the proof of Theorem \ref{th:wellp} (with details omitted for brevity),
we have:
\smallskip
\begin{thm}[Existence and uniqueness] Assume that $S^m$ satisfies:
\begin{itemize}
\item[(i)] $\min_{1 \leq j \leq N}|\sigma_j^m|>0$;
\item [(ii)] ${\rm dim}\{\vec n_{_\Gamma,_j}^m\}_{j=1}^{N_{_\Gamma}}=2$;
\item [(iii)] there exists a vertex $\vec q_{k_0}$ of the polygonal curve $\Gamma^m$ such that $\bmath{\omega}_{k_0}^m=(w^m_{k_0,1},~w_{k_0,2}^m,~w_{k_0,3}^m)^T$ satisfies $\left(w_{k_0,1}^m\right)^2+\left(w_{k_0,2}^m\right)^2>0$.
\end{itemize}
If $\cos\theta_Y=0$, i.e. $\theta_Y=\frac{\pi}{2}$, then the linear system in \eqref{eqn:anipfem} admits a unique solution.
\end{thm}

Denote
\begin{align}
\label{eq:anidisW}
W_\gamma^m:=\int_{S^m}\gamma(\vec n^m)\,\dA - \cos\theta_Y|S_1^m|=\sum_{j=1}^N\gamma(\vec n_j^m)|\sigma_j^m| - \cos\theta_Y|S_1^m|,
\end{align}
where $\gamma(\vec n)$ is defined in \eqref{eq:aniso}. Similar to the previous work in \cite{Barrett08Ani}, we can prove:
\begin{thm}\label{th:aniEnergyS} Let $\Bigl(\vec X^{m+1},~\mu^{m+1}\Bigr)$ be the numerical solution of \eqref{eqn:anipfem}, then the total energy of the system decrease in time, i.e.,
\begin{align}
W^{m+1}_\gamma  \leq W^m_\gamma \leq W^0_\gamma=\sum_{i=1}^L \sum_{j=1}^N\gamma_i(\vec n_j^0)|\sigma_j^0| -\cos\theta_Y|S_1^0|,\quad\forall m \geq 0.
\label{eq:aniEnergyS}
\end{align}
Moreover, we have
\begin{align}
&\sum_{l=1}^{m+1}\Vnorm{\nabla_{s}\mu^{l}}^2_{S^{l-1}} + \frac{1}{\eta}\sum_{l=1}^{m+1}\Vnorm{\left(\frac{\vec X^{l}_{_\Gamma} - \vec X^{l-1}_{_\Gamma}}{\tau}\right)\cdot\vec n^{l-1}_{_\Gamma}}^2_{\Gamma^{l-1}}  \leq \frac{W_\gamma^0-W^{m+1}_\gamma}{\tau},\quad \forall m  \geq 0,
\label{eq:aniEnergySW}
\end{align}
\end{thm}

\begin{proof}
Setting $\psi^h=\mu^{m+1}$ in \eqref{eq:anipfem1} and $\bmath{g}^h=\vec X^{m+1}-\vec X^m$ in \eqref{eq:anipfem2}, and combining these two equations we obtain
\begin{align}
\label{eq:anienergy1}
&\tau\,\Vnorm{\nabla_{s}\mu^{m+1}}_{S^m}^2 -\cos\theta_Y\Bigl(\vec n_{_\Gamma}^{m+\frac{1}{2}},~[\vec X_{_\Gamma}^{m+1} - \vec X_{_\Gamma}^m]\Bigr)_{\Gamma^m}+\frac{1}{\eta\,\tau}\Vnorm{(\vec X_{_\Gamma}^{m+1} - \vec X_{_\Gamma}^m)\cdot\vec n_{_\Gamma}^m}_{\Gamma^m}^2 \nn\\
&\hspace{2cm} + \sum_{i=1}^L\left(\gamma_i(\vec n^m),~\left(\nabla_{s}^{\widetilde{G}_i}\,\vec X^{m+1},~\nabla_{s}^{\widetilde{G}_i}\left(\vec X^{m+1}-\vec X^m\right)\right)_{\widetilde{G}_i}\right)_{S^m} =0.
\end{align}

Using the identity $a(a-b) = \frac{1}{2}\left(a^2-b^2+(a-b)^2\right)$, we obtain
\begin{align}
&\left(\gamma_i(\vec n^m),~\left(\nabla_{s}^{\widetilde{G}_i}\,\vec X^{m+1},~\nabla_{s}^{\widetilde{G}_i}\left(\vec X^{m+1}-\vec X^m\right)\right)_{\widetilde{G}_i}\right)_{S^m}\nn \\
\label{eq:anienergy2}
&=\frac{1}{2}\left(\gamma_i(\vec n^m),~\Vnorm{\nabla_{s}^{\widetilde{G}_i}\,\vec X^{m+1}}_{\widetilde{G}}^2 - \Vnorm{\nabla_{s}^{\widetilde{G}_i}\,\vec X^{m}}_{\widetilde{G}}^2 + \Vnorm{\nabla_{s}^{\widetilde{G}_i}\left(\vec X^{m+1}-\vec X^m\right)}_{\widetilde{G}}^2\right)_{S^m}\\
& \geq \int_{S^{m+1}}\gamma_i(\vec n^{m+1})\,\dA - \int_{S^m}\gamma_i(\vec n^m)\,\dA,\nn
\end{align}
where $\Vnorm{\cdot}_{\widetilde{G}}$ is the induced norm of the inner product in \eqref{eq:aniinner}, and the last inequality is a direct application of \eqref{eq:asoieq}:
\begin{subequations}
\begin{align}
&\frac{1}{2}\left(\gamma_i(\vec n^m),~\Vnorm{\nabla_{s}^{\widetilde{G}_i}\,\vec X^{m+1}}_{\widetilde{G}}^2\right)_{S^m} \geq  \int_{S^{m+1}}\gamma_i(\vec n^{m+1})\,\dA,\\
&\frac{1}{2}\left(\gamma_i(\vec n^m),~\Vnorm{\nabla_{s}^{\widetilde{G}_i}\,\vec X^{m}}_{\widetilde{G}}^2\right)_{S^m} = \int_{S^{m}}\gamma_i(\vec n^{m})\,\dA.
\end{align}
\end{subequations}
Substituting \eqref{eq:anienergy2},  \eqref{eq:aniso} and \eqref{eq:anidisW} into \eqref{eq:anienergy1}, and also noting the equality \eqref{eq:subenergy} yields 
\begin{align}
W^{m+1}_\gamma + \tau\Vnorm{\nabla_{s}\mu^{m+1}}_{S^m}^2 + \frac{1}{\eta\tau}\Vnorm{\left(\vec X_{_\Gamma}^{m+1} - \vec X_{_\Gamma}^m\right)\cdot\vec n_{_\Gamma}^m}^2_{\Gamma^m} \leq W_\gamma^m,\qquad m \geq 0.
\label{eq:anienergy4}
\end{align}

Reformulating \eqref{eq:anienergy4} as
\begin{align*}
\Vnorm{\nabla_{s}~\mu^l}_{S^{l-1}}^2 + \frac{1}{\eta}\,\Vnorm{\left(\frac{\vec X_{_\Gamma}^{l} - \vec X_{_\Gamma}^{l-1}}{\tau}\right)\cdot\vec n_{_\Gamma}^{l-1}}_{\Gamma^{l-1}}^2 \leq\frac{W_\gamma^{l-1}-W_\gamma^{l}}{\tau},\qquad l \geq 1,
\end{align*}
summing up for $l=1,2,\cdots,m+1$ yields \eqref{eq:aniEnergySW}.
\end{proof}

\section{Numerical results}
\label{sec:numr}

In this section, we first present the convergence tests of our method in \eqref{eqn:isopfem} and \eqref{eqn:anipfem}, and then report numerical examples to demonstrate the morphological features in SSD. The resulting linear system in \eqref{eqn:isopfem} and \eqref{eqn:anipfem} can be efficiently solved via the Schur complement method discussed in \cite{Barrett08JCP} or simply GMRES method with preconditioners based on incomplete LU factorization. The contact line mobility $\eta$ in \eqref{eq:bd2} controls the relaxation rate of the dynamic contact angle to the equilibrium angle, and large $\eta$ accelerates the relaxation process \cite{Wang15, Zhao19b}. In what follows, we will always choose $\eta=100$ unless otherwise stated.

\subsection{Convergence tests}
\label{ssec:convet}

\begin{figure}[!htp]
\centering
\includegraphics[width=0.7\textwidth]{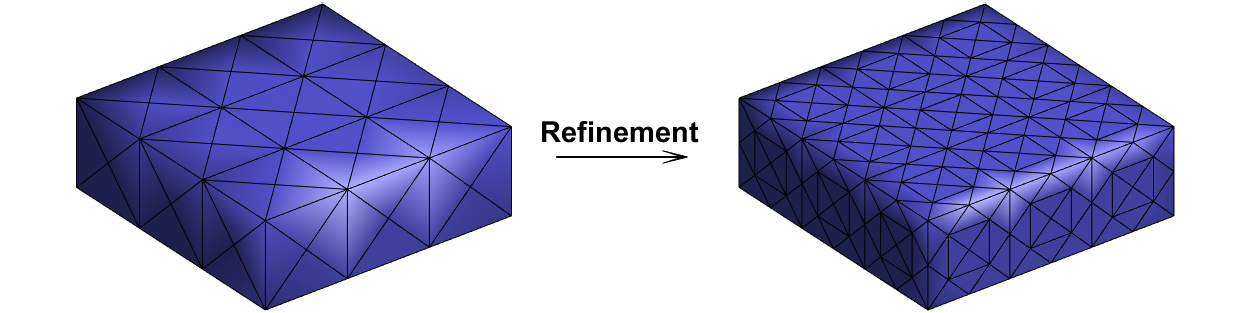}
\caption{The interface partition of a $(3,3,1)$ cuboid with mesh size $h=0.5$ (left panel), and a refined polygonal surface mesh with mesh size $h=0.25$ (right panel) obtained by subdividing each triangle into 4 smaller triangles.}
\end{figure}

We test the convergence rate of the numerical methods in \eqref{eqn:isopfem} and \eqref{eqn:anipfem} by carrying out numerical simulations under different mesh sizes and time step sizes. Initially, the island film is chosen as a cuboid island with (3, 3, 1) representing its length, width and height. The region occupied by the initial thin film is then given by $[-1.5,~1.5]\times[-1.5,~ 1.5]\times[0,~1]$. Let $S^0:=\cup_{j=1}^N\overline{\sigma_j^0}$ be the initial partition of $S(0)$ and $\vec X_{h,\tau}(\cdot, t_m)$ be the numerical solution of the interface $S(t_m)$ obtained using mesh size $h=\max_{j=1}^N\sqrt{|\sigma_j^0|}$ and time step size $\tau$. Then we measure the error of the numerical solutions by comparing $\vec X_{h,\tau}$ and $\vec X_{\frac{h}{2},\frac{\tau}{4}}$.

To measure the difference between two polygonal surfaces given by
\begin{subequations}
\begin{align}
S:=\cup_{j=1}^N\overline{\sigma_j}\quad{\rm with\; vertices} \quad \left\{\vec q_k\right\}_{k=1}^K,\\
S^\prime=\cup_{j=1}^{N^\prime}\overline{\sigma^\prime_j}\quad{\rm  with\;vertices} \quad \left\{\vec q_k^\prime\right\}_{k=1}^{K^\prime},
\end{align}
\end{subequations}
we define the following manifold distance
\begin{align}
\label{eq:manifold}
\mathcal{M}\left(S,~S^\prime\right) = \frac{1}{2}\left(\max_{1 \leq k \leq K^\prime}\min_{1 \leq j \leq N}\,{\rm dist\left(\vec q_k^\prime,~\sigma_j\right)} + \max_{1 \leq k \leq K}\min_{1 \leq j \leq N^\prime}\,{\rm dist\left(\vec q_k,~\sigma_j^\prime\right)}\right),
\end{align}
where ${\rm dist}(\vec q,~\sigma)=\inf_{\vec p\in\sigma}\Vnorm{\vec p -\vec q}$ represents the distance of the vertex $\vec q$ to the triangle $\sigma$.

We fix $\theta_Y=2\pi/3$ and consider the isotropic surface energies and ellipsoidal surface energies in \eqref{eq:ellip} with $a_1=2,a_2=1,a_3=1$. The numerical errors are computed based on the manifold distance \eqref{eq:manifold} by
\begin{align}
e_{h,\tau}(t_m):=\mathcal{M}\left(\vec X_{h,\tau},~\vec X_{\frac{h}{2},\frac{\tau}{4}}\right),\quad m \geq 0.
\end{align}
Numerical errors for the two cases are reported in Table \ref{tb:isoerror}.  We observe that the error decreases with refined mesh size and time step size, and the order of convergence can reach about 2 for spatial discretization.

\begin{table}[!t]
\centering
\def\temptablewidth{0.85\textwidth}
\vspace{-1pt}
\caption{Error $e_{h,\tau}$ and the rate of convergence for the dynamic interface under the isotropic surface energy (upper panel) and the ellipsoidal surface energy \eqref{eq:ellip} with $a_1=2,a_2=1,a_3=1$ (lower panel) at three different times. Other parameters are chosen as $\theta_Y=2\pi/3$, $h_0=0.5$ and $\tau_0=0.01$.}
\label{tb:isoerror}
{\rule{\temptablewidth}{1pt}}
\begin{tabular*}{\temptablewidth}{@{\extracolsep{\fill}}c|cc|cc|cc}
$(h,\,\tau)$ &$e_{h,\tau}(t=0.5)$ & order &$e_{h,\tau}(t=1.0)$
& order &$e_{h,\tau}(t=2.0)$ & order  \\ \hline
$(h_0,\,\tau_0)$ & 8.17E-2 & - &7.19E-2 &-& 6.61E-2 &- \\ \hline
$(\frac{h_0}{2},\, \frac{\tau_0}{2^2})$ & 2.05E-2 & 1.99 &1.71E-2 &2.07& 1.71E-2 &1.95
\\ \hline
$(\frac{h_0}{2^2},\,\frac{\tau_0}{2^4})$ & 4.80E-3 & 2.07 &4.85E-3 &1.82& 5.20E-3 &1.72
 \end{tabular*}
{\rule{\temptablewidth}{1pt}}
{\rule{\temptablewidth}{1pt}}
\begin{tabular*}{\temptablewidth}{@{\extracolsep{\fill}}c|cc|cc|cc}
$(h,\,\tau)$ &$e_{h,\tau}(t=0.5)$ & order &$e_{h,\tau}(t=1.0)$
& order &$e_{h,\tau}(t=2.0)$ & order  \\ \hline
$(h_0,\,\tau_0)$ & 8.03E-2 & - &7.93E-2 &-& 7.85E-2 &- \\ \hline
$(\frac{h_0}{2},\, \frac{\tau_0}{2^2})$ & 2.03E-2 & 1.98 &2.15E-2 &1.88& 2.25E-2 &1.80
\\ \hline
$(\frac{h_0}{2^2},\,\frac{\tau_0}{2^4})$ & 5.31E-3 & 1.93 &5.46E-3 &2.09& 5.45E-3 &2.05
 \end{tabular*}
{\rule{\temptablewidth}{1pt}}
\end{table}

\begin{figure}[tph]
\centering
\includegraphics[width=0.95\textwidth]{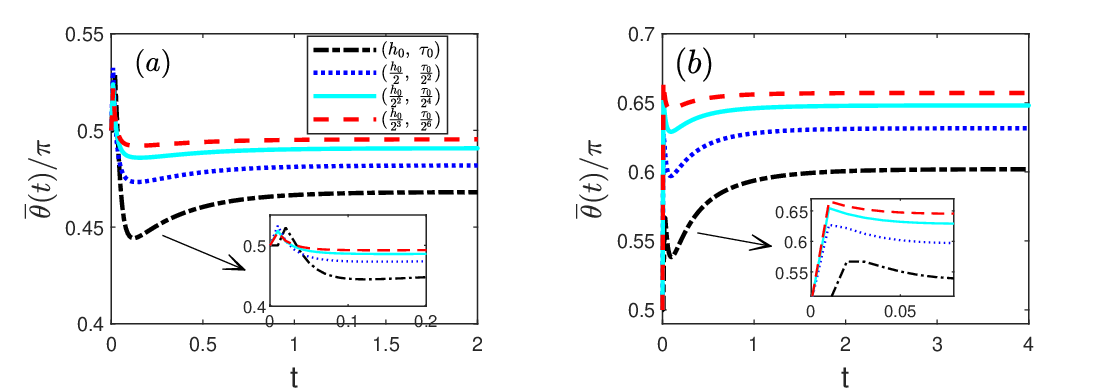}
\caption{The time history of the average contact angle defined in \eqref{eq:avT} by using different mesh sizes and time step sizes, where $h_0=0.5$ and $\tau_0=0.01$. (a) $\theta_Y=\pi/2$; (b) $\theta_Y=2\pi/3$.}
\label{fig:angleE}
\end{figure}

 \begin{table}[tph]
 \centering
\def\temptablewidth{0.70\textwidth}
\vspace{-1pt}
\caption{Numerical errors between the numerical equilibrium contact angle (at time $t=10$) and the Young's angle $\theta_Y$ for isotropic surface energy, where $h_0=0.5$, $\tau_0=0.01$.}
\label{tb:angle}
{\rule{\temptablewidth}{1pt}}
\begin{tabular*}{\temptablewidth}{@{\extracolsep{\fill}}c|cc|cc}
 \multirow{2}{*}{$(h, \,\tau)$} & \multicolumn{2}{c|}{$\theta_Y=\pi/2$}
&\multicolumn{2}{c}{$\theta_Y=2\pi/3$}\\ \cline{2-5}
  &$|\overline{\theta}-\theta_Y|(t=10) $ & order &$|\overline{\theta}-\theta_Y|(t=10)$
& order \\ \hline
$(h_0,\,\tau_0)$ & 1.00E-1 & - &2.03E-1 &- \\ \hline
$(\frac{h_0}{2},\, \frac{\tau_0}{2^2})$ & 5.70E-2 & 0.81 &1.10E-1 &0.88
\\ \hline
$(\frac{h_0}{2^2},\,\frac{\tau_0}{2^4})$ & 2.90E-2 & 0.97 &5.72E-2 &0.95 \\ \hline
$(\frac{h_0}{2^3},\,\frac{\tau_0}{2^6})$ & 1.44E-2 & 1.05 &2.98E-2 &0.94 \\
 \end{tabular*}
{\rule{\temptablewidth}{1pt}}
\end{table}

To further assess the accuracy of the numerical method in \eqref{eqn:isopfem}, we define the following average contact angle $\overline{\theta}$
\begin{align}
\label{eq:avT}
\left.\overline{\theta}\right|_{t=t_m}:=\frac{1}{N_{_\Gamma}}\sum_{j=1}^{N_{_\Gamma}}\arccos\left(\vec c_{_{\Gamma,j}}^m\cdot\vec n_{_{\Gamma,j}}^m\right),
\end{align}
where $\vec c_{_\Gamma, _j}^m$ and $\vec n_{_\Gamma, _j}^m$ are numerical approximations of $\vec c_{_\Gamma}$ and $\vec n_{_\Gamma}$ at $j$th line segment of $\Gamma^m$, respectively. The time history of the average contact angles computed using different mesh sizes and time step sizes for $\theta_Y=\pi/2$ and $\theta_Y=2\pi/3$ are presented in Fig.~\ref{fig:angleE}. We observe the convergence of the dynamic contact angle as the mesh is refined. A more quantitative assessment for the average contact angel $\overline{\theta}$ is provided in Table.~\ref{tb:angle}, where we show the error between $\overline{\theta}$ for the equilibrium state ($t=10$) and the Young's angle $\theta_Y$. We observe the error decreases as mesh is refined, and the convergence rate for $|\overline{\theta}-\theta_Y|$ is about 1.

\subsection{Applications}
\label{ssec:app}

\begin{figure}[!t]
\centering
\includegraphics[width=0.8\textwidth]{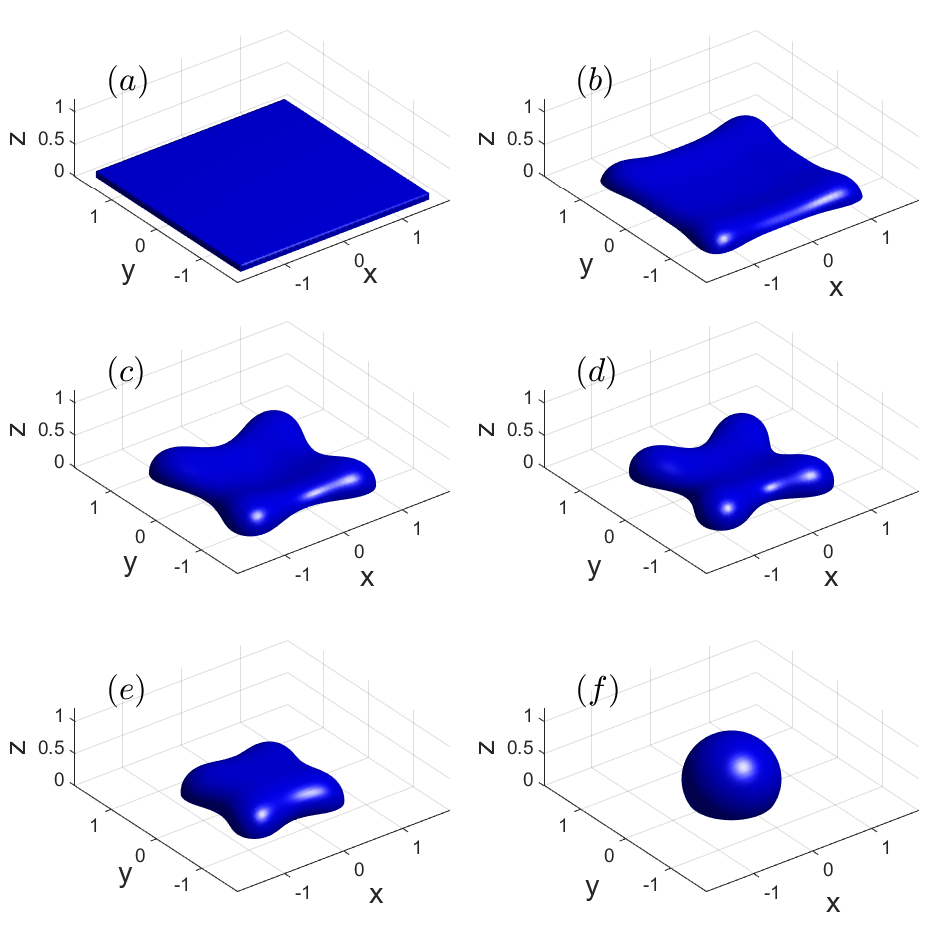}
\caption{Several snapshots in the evolution of an initially square island film towards its equilibrium shape under the isotropic surface energy: (a) $t = 0$; (b) $t = 0.004$; (c) $t = 0.008$; (d)
$t = 0.012$; (e) $t = 0.020$; (f) $t = 0.080$. The initial shape is chosen as a (3.2, 3.2, 0.1) cuboid,
and $\theta_Y=2\pi/3$.}
\label{fig:isoe}
\end{figure}

We present several numerical examples to demonstrate the anisotropic effects on the morphological evolution of thin films in SSD.

\vspace{0.2cm}
\noindent {\bf Example 1.}
In this example, we consider the evolution of square island films under: (1) isotropic surface energy $\gamma(\vec n)=1$; (2) ellipsoidal surface energy in \eqref{eq:ellip} with $a_1=2, a_2=1, a_3=1$; and (3) ``cusped'' surface energy in \eqref{eq:cuspgamma} with $\delta = 0.1$. The initial film is chosen as a $(3.2,3.2,0.1)$ cuboid island. The interface is partitioned into $N=18432$ triangles with $K=9345$ vertices, and we take $\tau= 1\times 10^{-4}$, $\theta_Y=2\pi/3$.

Several snapshots of the morphological evolutions for the thin film under the three anisotropies are shown in Fig.~\ref{fig:isoe}, \ref{fig:ellie} and \ref{fig:cuspe}, respectively. In the isotropic case, we observe the four corners of the square island retract much more slowly than that of the four edges at the beginning, thus forming a cross-shaped geometry. As time evolves, the island finally forms a perfectly spherical shape as equilibrium. For ellipsoidal surface energy, the square island forms a grooved shape and then reach an ellipsoidal shape as equilibrium. In the case of ``cusped'' surface energy, the island maintains a self-similar cuboid shape by gradually decreasing its  length, width and increasing its height until reaching the equilibrium state.

\begin{figure}[!t]
\centering
\includegraphics[width=0.8\textwidth]{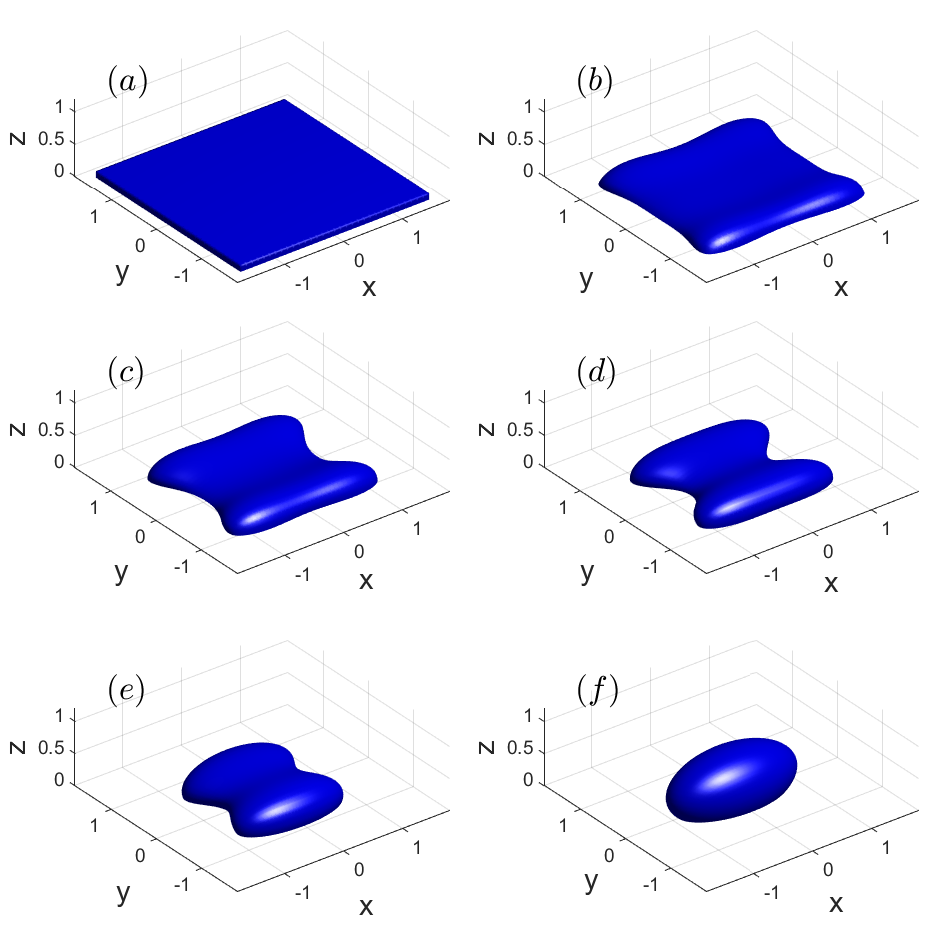}
\caption{Several snapshots in the evolution of an initially square island film towards its equilibrium shape under the ellipsoidal surface energy: (a) $t = 0$; (b) $t = 0.004$; (c) $t = 0.008$; (d)
$t = 0.012$; (e) $t = 0.020$; (f) $t = 0.080$. The initial shape is chosen as a (3.2, 3.2, 0.1) cuboid,
and $\theta_Y=2\pi/3$. The surface energy is chosen in \eqref{eq:ellip} with $a_1=2, a_2=1,a_3=1$.}
\label{fig:ellie}
\end{figure}

\begin{figure}[tph]
\centering
\includegraphics[width=0.8\textwidth]{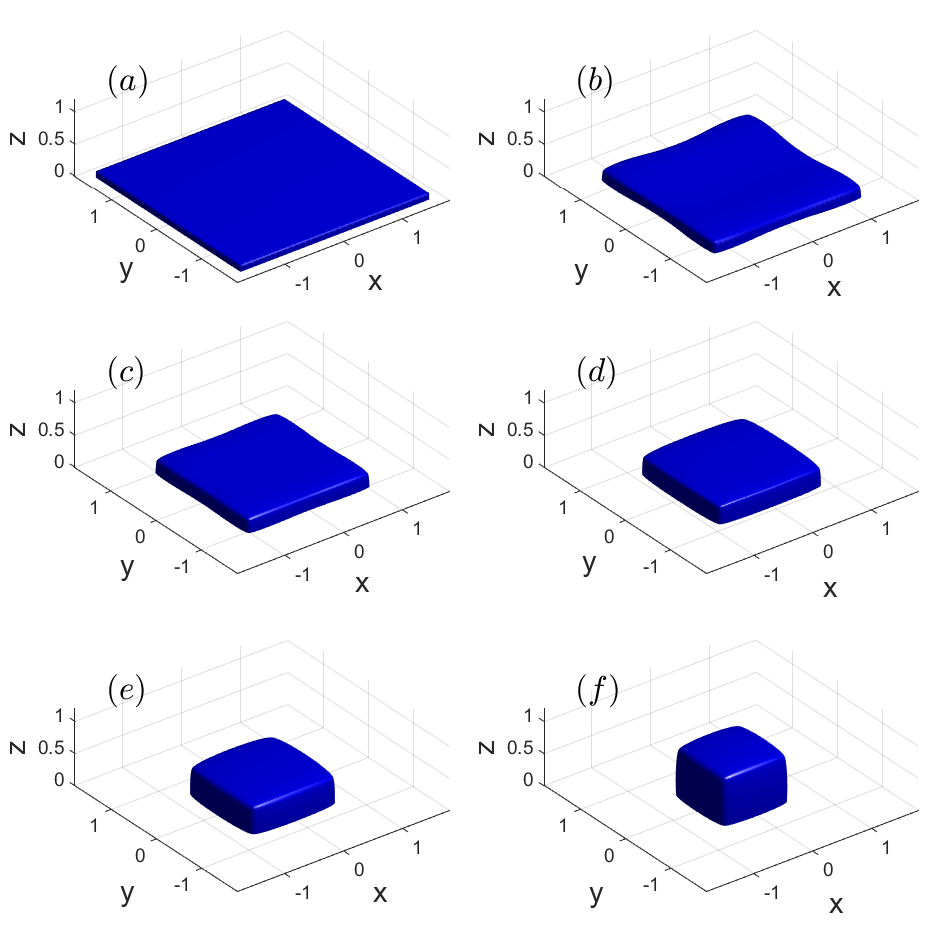}
\caption{Several snapshots in the evolution of an initially square island film towards its equilibrium shape under the ``cusped'' surface energy: (a) $t = 0$; (b) $t = 0.004$; (c) $t = 0.008$; (d)
$t = 0.012$; (e) $t = 0.020$; (f) $t = 0.080$, where the initial shape is chosen as a (3.2, 3.2, 0.1) cuboid,
and $\theta_Y=2\pi/3$. The surface energy is chosen in \eqref{eq:cuspgamma} with $\delta=0.1$.}
\label{fig:cuspe}
\end{figure}

\begin{figure}[!tph]
\centering
\includegraphics[width=0.8\textwidth]{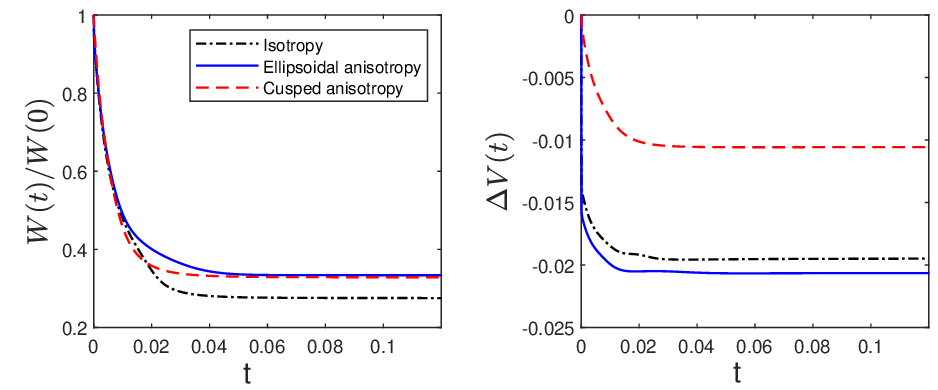}
\caption{The temporal evolution of the normalized energy $W(t)/W(0)$ (left panel) and the relative volume loss $\Delta V(t)$ (right panel) by using the isotropic surface energy, the ellipsoidal surface energy and cusped surface energy.}
\label{fig:WVE}
\end{figure}

In Fig.~\ref{fig:WVE}, we plot the temporal evolution of the normalized energy  $W(t)/W(0)$ and the relative volume loss $\Delta V$ defined as
\begin{align}
\Delta V:=\frac{V^h(t)-V^h(0)}{|V^h(0)|},\qquad t \geq 0,
\end{align}
where $V^h(0)$ is the total volume of the initial shape. We observe the total free energy of the discrete system decays in time, and the relative volume loss is about 1\% to 2\%. We note recently a structure-preserving method that can conserve the enclosed volume exactly for surface diffusion was presented in \cite{Zhao2021}, and the generalization to axisymmetric geometric equations was considered in \cite{Bao2022volume}. In addition, an energy-stable method for anisotropic surface diffusion under general anisotropies has been discussed in \cite{Li2020energy}.

\vspace{0.2cm}
\noindent {\bf  Example 2.}  We investigate the equilibrium shapes of SSD by using different $\theta_Y$ and surface energies. We consider the ``cusped'' surface energy defined in \eqref{eq:cuspgamma} with four different rotations: (1) $\gamma(\vec n)$ defined in \eqref{eq:cuspgamma} with $\delta=0.1$; (2) $\gamma\left(\bmath{R}_x(\pi/4)\,\vec n\right)$; (3) $\gamma\left(\bmath{R}_y(\pi/4)\,\vec n\right)$; (4) $\gamma\left(\bmath{R}_z(\pi/4)\,\vec n\right)$, where $\bmath{R}_{x}(\pi/4)$, $\bmath{R}_y(\pi/4)$ and $\bmath{R}_z(\pi/4)$ represent the orthogonal matrix for the rotation by an angle $\pi/4$ about the x,y,z-axis using the right-hand rule, respectively. The initial thin film is chosen as a $(3,3,1)$ cuboid island. The interface is partitioned into $N=5376$ triangles with $K=2737$ vertices, and $\tau=10^{-2}$.

As it can be seen from Fig.~\ref{fig:equilibrium}(a)-(c), $\theta_Y$ controls the equilibrium contact angle and thus the equilibrium shape. From Fig. \ref{fig:equilibrium}(d)-(f), we observe that a rotation of the ``cusped'' anisotropy will result in a corresponding rotation of the equilibrium shape.

\begin{figure}[!htp]
\centering
\includegraphics[width=0.9\textwidth]{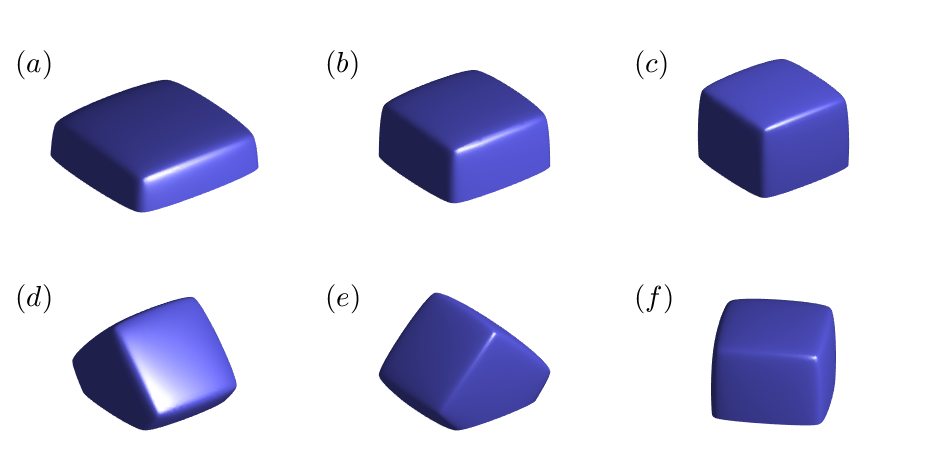}
\caption{The equilibrium profiles of the island film obtained using different $\theta_Y$ and anisotropies. For (a)-(c), we fix $\gamma(\vec n)$ in \eqref{eq:cuspgamma} with $\delta=0.1$, and $\theta_Y=\pi/3,~\pi/2,~2\pi/3$, respectively. For (d) -(f), we set $\theta_Y=2\pi/3$ but choose $\gamma\left(\bmath{R}_x(\pi/4)\,\vec n\right)$,  $\gamma\left(\bmath{R}_y(\pi/4)\,\vec n\right)$ and $\gamma\left(\bmath{R}_z(\pi/4)\,\vec n\right)$, respectively.}
\label{fig:equilibrium}
\end{figure}

\vspace{0.2cm}
\noindent {\bf  Example 3.}
We examine the geometric evolution of the square-ring patch under the ``cusped'' surface energies used in Fig.~\ref{fig:equilibrium}(c)-(f). The initial island is chosen as a $(12, 12, 1)$ cuboid by cutting out a (10, 10, 1) cuboid from the centre. The interface is partitioned into $N=33792$ triangles with $K=17248$ vertices, and we take $\tau = 2\times 10^{-4}$,  $\theta_Y=2\pi/3$.

\begin{figure}[!htp]
\centering
\includegraphics[width=0.85\textwidth]{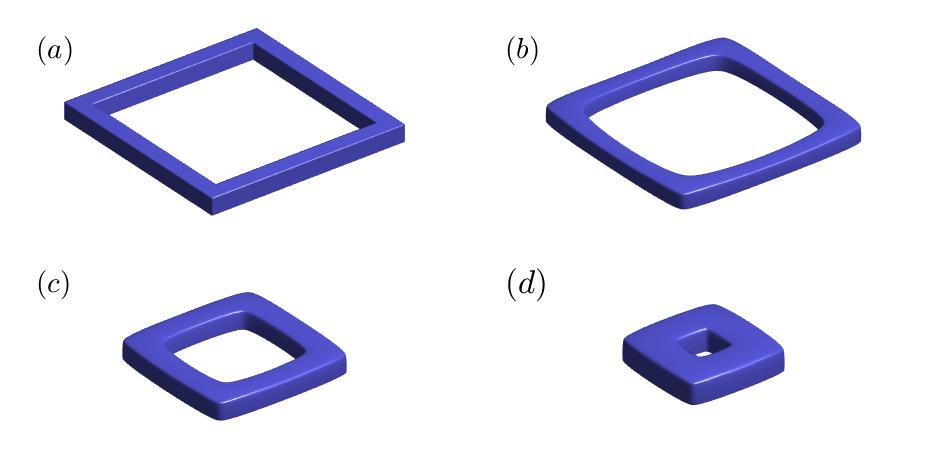}
\caption{Snapshots of the island film in the evolution of an initially square-ring patch with $\theta_Y=2\pi/3$. The surface energy is given by \eqref{eq:cuspgamma} with $\delta=0.1$. (a) $t=0$; (b) $t=1.0$; (c) $t=8.0$; (d) $t=12.6$.}
\label{fig:Ani}
\end{figure}

\begin{figure}[!htp]
\centering
\includegraphics[width=0.8\textwidth]{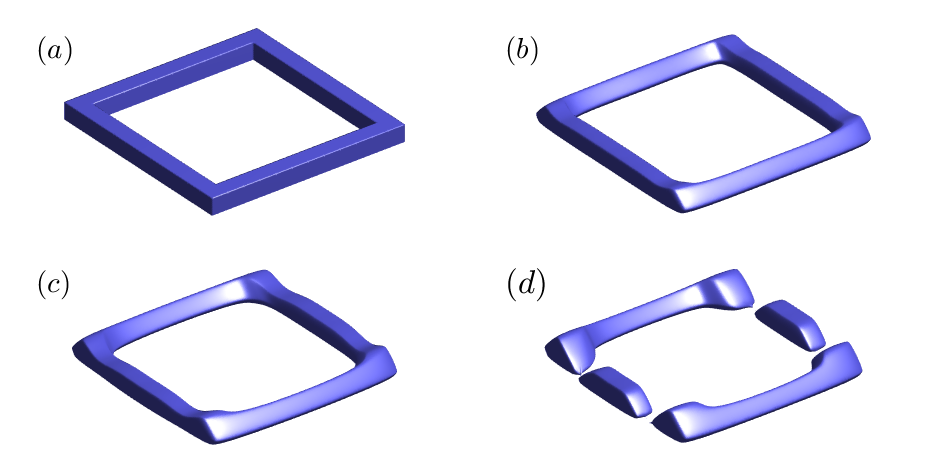}
\caption{Snapshots of the island film in the evolution of an initially square-ring patch with $\theta_Y=2\pi/3$. The surface energy is given by $\gamma(\bmath{R}_x(\pi/4)\,\vec n)$, where $\gamma(\vec n)$ is defined in \eqref{eq:cuspgamma} with $\delta=0.1$. (a) $t=0$; (b) $t=0.1$; (c) $t=0.5$; (d) $t=1.12$.}
\label{fig:Anix}
\end{figure}

\begin{figure}[!htp]
\centering
\includegraphics[width=0.8\textwidth]{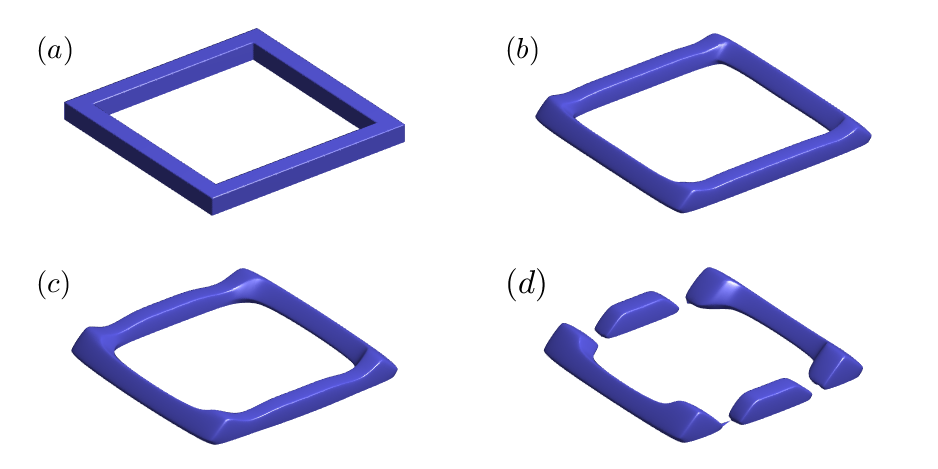}
\caption{Snapshots of the island film in the evolution of an initially square-ring patch with $\theta_Y=2\pi/3$. The surface energy is chosen as $\gamma(\bmath{R}_y(\pi/4)\,\vec n)$, where $\gamma(\vec n)$ is defined in \eqref{eq:cuspgamma} with $\delta=0.1$. (a) $t=0$; (b) $t=0.1$; (c) $t=0.5$; (d) $t=1.12$.}
\label{fig:Aniy}
\end{figure}

\begin{figure}[!htp]
\centering
\includegraphics[width=0.8\textwidth]{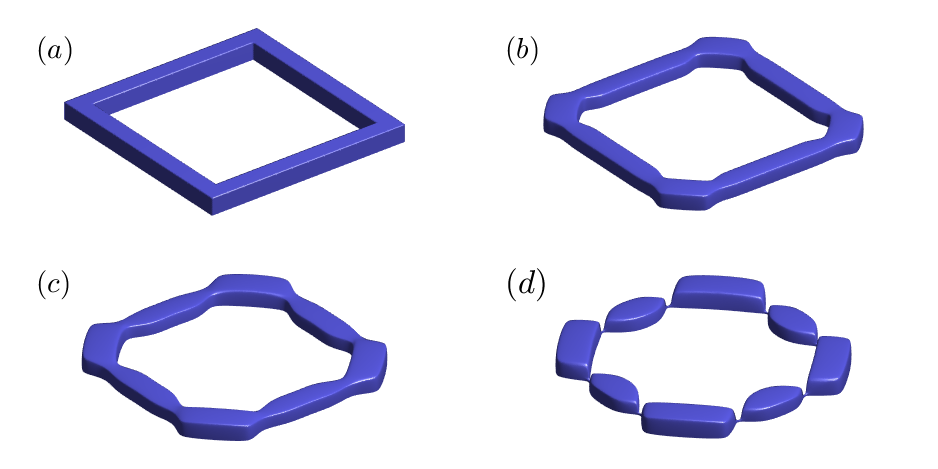}
\caption{Snapshots of the island film in the evolution of an initially square-ring patch with $\theta_Y=2\pi/3$. The surface energy is given by $\gamma(\bmath{R}_z(\pi/4)\,\vec n)$, where $\gamma(\vec n)$ is defined in \eqref{eq:cuspgamma} with $\delta=0.1$.  (a) $t=0$; (b) $t=0.1$; (c) $t=0.4$; (d) $t=0.76$.}
\label{fig:Aniz}
\end{figure}

The geometric evolutions for the square-ring island are shown in Fig.~\ref{fig:Ani} -\ref{fig:Aniz} for the four cases, respectively.  When the surface energy density is given by \eqref{eq:cuspgamma}, we observe that the thin square-ring patch gradually shrinks towards the centre to form a self-similar square-ring shape. When the orientation of the anisotropy is rotated with respect to $x$-axis by $\pi/4$, we observe from Fig.~\ref{fig:Anix} that a break-up of the island along the $y$-direction occures. Similarly, a rotation of the anisotropy with respect to $y$-axis results in the breakup of the island along $x$-axis, as expected in Fig.~\ref{fig:Aniy}. Furthermore, we observe that when the orientation of the anisotropy is rotated with respect to the $z$-axis, the square-ring  forms several isolated particles.  In these experiments, topological change event could occur when the island is breaking up into small particles. In this situation, we simply show the results before the blowup of the numerical solutions.

It is well-known in isotropic case, the Rayleigh-like instability in the azimuthal direction and the shrinking instability in the radial direction are competing with each other to determine the geometric evolution of a square-ring island~\cite{Jiang19b, Zhao19b}. Generally, a very thin square island always breaks up into isolated particles since the Rayleigh-like instability dominates the kinetic evolution; while for a fat square-ring island,  the shrinking instability dominates the evolution and makes the island shrink towards the center. Our numerical simulations indicate that anisotropic surface energies play an important role in the evolution of the square-ring island. They can either enhance or mitigate the Rayleigh-like instability in the azimuthal direction, depending on the crystalline alignments. 

\section{Conclusions}
\label{sec:con}
We developed an efficient, accurate, and energy-stable parametric finite element method (ES-PFEM) for solving the sharp-interface model of solid-state dewetting in both the isotropic case and the anisotropic case with anisotropic surface energies in the Riemannian metric form. By reformulating the relaxed contact angle condition as a time-dependent Robin-type of boundary condition for the interface, we obtained a new weak formulation. By using continuous piecewise linear elements in space and the backward Euler method in time, we then discretized the weak formulation to obtain the semi-implicit ES-PFEM. We proved the well-posedness and unconditional energy stability of the numerical method.

We assessed the accuracy and convergence of the proposed ES-PFEM and found that it can attain the second-order convergence rate of the spatial error for the dynamic interface and the first-order convergence rate of the contact angle for the equilibrium interface. Finally, we investigated the anisotropic effects on the evolution of large square islands and square-ring islands by using different anisotropic surface energies. In fact, the proposed ES-PFEM provides a nice tool for  simulating different applications in solid-state dewetting in three dimensions.

\begin{appendices}

\section{Anisotropic surface gradient}\label{sec:app1}
 Given a two-dimensional manifold $S$ and a smooth scalar field $f$, the surface gradient of $f$ over $S$ is defined as
\begin{align}
\label{eq:sfoperator}
\nabla_{s}f(\vec X)=\left(\mathbb{I}-\vec n\otimes\vec n\right)\nabla f(\vec X):=(\underline{D}_1f(\vec X),~\underline{D}_2f(\vec X),~\underline{D}_3f(\vec X))^T,\quad\vec X\in S,
\end{align}
where $\mathbb{I}\in\mathbb{R}^{3\times 3}$ is the identity matrix and $\vec n$ is the unit normal to $S$. 

Let $G\in\mathbb{R}^{3\times 3}$ be a symmetric positive definite (SPD) matrix, we follow the notations in \cite{Barrett08Ani} and define the anisotropic surface gradient associated with the SPD matrix $\widetilde{G} = ({\rm det}{G})^{-\frac{1}{2}} G$ as
\begin{align}
\label{eq:anigradient}
\nabla_{s}^{\widetilde{G}} f (\vec X)= \sum_{i=1}^2(\partial_{\vec t_i}f)(\vec X)\, \vec t_i,\qquad\vec X\in S,
\end{align}
where $\partial_{\vec t_i}f = \vec t_i\cdot(\nabla_{s}\,f)$ is the directional derivative,  and $\left\{\vec t_1,~\vec t_2\right\}$ is the orthogonal basis of the tangential space of $S$ with respect to $\widetilde{G}$ at the point of interest $\vec X$, i.e.,
\begin{align}
\label{eq:ortho}
\vec t_i \cdot(\widetilde{G}\,\vec t_j) = \delta_{ij},\quad \vec t_i\cdot\vec n(\vec X) = 0,\qquad i,j=1,2.
\end{align}
 Moreover, given a vector-valued smooth function $\bmath{g}$, the anisotropic surface divergence and anisotropic surface gradient are then defined respectively as
 \begin{subequations}
 \begin{align}
 (\nabla_{s}^{\widetilde{G}}\cdot\bmath{g})(\vec X) = \sum_{i=1}^2 (\partial_{\vec t_i}\bmath{g})(\vec X)\cdot({\widetilde{G}}\,\vec t_i),\\
 (\nabla_{s}^{\widetilde{G}}\,\bmath{g})(\vec X) = \sum_{i=1}^2(\partial_{\vec t_i}\,\bmath{g})(\vec X)\otimes(\widetilde{G}\vec t_i).
 \end{align}
 \end{subequations}
 
\section{Differential calculus}\label{sec:app2}

Given the surface energy density $\gamma(\vec n)$ in \eqref{eq:aniso} and the moving surface $S(t)$ with boundary $\Gamma(t)$, then we have the following equation hold (see Lemma 2.1 in \cite{Jiang19c}):
\begin{align}
\label{eq:anivar}
\frac{\rd }{\rd t}\int_{S(t)}\gamma(\vec n)\,\rd A = \int_{S(t)}\mu\,\vec n\cdot\bmath{v}\,\rd A +\int_{\Gamma(t)}\vec c_{_\Gamma}^\gamma\cdot\bmath{v}\,\rd s,
\end{align}
where $\mu$ is the chemical potential defined in \eqref{eq:animodel2} and $\bmath{v}$ is the velocity of $S(t)$.

 Besides, we have (see Lemma 3.2 in \cite{Barrett10cluster}):
\begin{align}
\label{eq:anirey}
\frac{\rd }{\rd t}\sum_{i=1}^L\int_{S(t)}\gamma_i(\vec n)\,\rd A:=\sum_{i=1}^L\int_{S(t)}\gamma_i(\vec n)\left(\nabla_{s}^{\widetilde{G}_i}\,\vec X,~\nabla_{s}^{\widetilde{G}_i}\bmath{v}\right)_{\widetilde{G}_i}\,\rd A,
\end{align}
where we define the inner product with respect to a particular matrix $\widetilde{G}$ for $\vec u,~\vec v$ via
\begin{align}
\label{eq:aniinner}
\Bigl(\nabla_{s}^{\widetilde{G}}\vec u,~\nabla_{s}^{\widetilde{G}}\vec v\Bigr)_{\widetilde{G}} = \sum_{i=1}^2\Bigl(\partial_{\vec t_{i}}\vec u,~\partial_{\vec t_{i}}\vec v\Bigr)_{\widetilde{G}} = \sum_{i=1}^2\partial_{\vec t_{i}}\vec u\cdot(\widetilde{G}\,\partial_{\vec t_{i}}\vec v),
\end{align}
with $\{\vec t_{1},~\vec t_2\}$ satisfying \eqref{eq:ortho}. In particular, when $L=1$ and $\widetilde{G}_1=G_1=\mathbb{I}$, Eq.~\eqref{eq:anirey} will reduce to the Reynolds transport theorem on $S(t)$
\begin{align}
\label{eq:isorey}
\frac{\rd }{\rd t}\left|S(t)\right| = \int_{S(t)}\nabla_{s}\vec X:\nabla_{s}\bmath{v}\,\dA.
\end{align}

Combining \eqref{eq:anivar} and \eqref{eq:anirey} yields
\begin{align}
\label{eq:muform2}
&\int_{S(t)}\mu\,\vec n\cdot\bmath{v}\,\rd A -\sum_{i=1}^L\int_{S(t)}\gamma_i(\vec n)\,\left(\nabla_{s}^{\widetilde{G}_i}\,\vec X,~\nabla_{s}^{\widetilde{G}_i}\bmath{v}\right)_{\widetilde{G}_i}\,\rd A+\int_{\Gamma(t)}\vec c_{_\Gamma}^\gamma\cdot\bmath{v}\,\rd s = 0.
\end{align}
When $L=1$ and $\widetilde{G}_1=G_1=\mathbb{I}$, Eq.~\eqref{eq:muform2} will reduce to
\begin{align}
\label{eq:Kformu}
\int_{S(t)} \mathcal{H}\,\vec n\cdot\bmath{v}\,\rd A-\int_{S(t)}\nabla_{s}\vec X:\nabla_{s}\bmath{v}\,\rd A
+\int_{\Gamma(t)}\vec c_{_\Gamma}\cdot\bmath{v}\,\rd s = 0.
\end{align}

\section{Relevant inequalities}

Given the surface energy density $\gamma(\vec n)$ in \eqref{eq:aniso}, and the polygonal surface $S^m:=\vec X^m(\cdot)\in\mathbb{X}^m$ in \eqref{eq:polygonalS}, where $\mathbb{X}^m$ is defined in \eqref{eqn:FEMspaces}. For $\vec Y\in\mathbb{X}^m$, we have (see Lemma 3.1 in \cite{Barrett08Ani}):
\begin{align}
\label{eq:asoieq}
\frac{1}{2}\int_{\sigma_j^m}\gamma_i(\vec n^m)\left(\nabla_{s}^{\widetilde{G}_i}\,\vec Y,~\nabla_{s}^{\widetilde{G}_i}\vec Y\right)_{\widetilde{G}_i}\,\rd A\geq \int_{\vec Y(\sigma_j^m)} \gamma_i(\vec n(\vec Y))\,\rd A,\quad 1\leq i\leq L,\quad 1\leq j\leq N,
\end{align}
and the equality holds when $\vec Y = \vec X^m$. When $\gamma_i\equiv 1$, i.e, $G_i=\widetilde{G}_i=\mathbb{I}$, we obtain for $1\leq j\leq N$
\begin{subequations}
\label{eqn:isoieq}
\begin{align}
&\frac{1}{2}\int_{\sigma_j^m}\nabla_s\vec Y:\nabla_s\vec Y\,\rd A\geq \int_{\vec Y(\sigma_j^m)}\,\rd A=|\vec Y(\sigma_j^m)|,\\
&\frac{1}{2}\int_{\sigma_j^m}\nabla_s\vec X^m:\nabla_s\vec X^m\,\rd A = \int_{\sigma_j^m}\,\rd A = |\sigma_j^m|.
\end{align}
\end{subequations}
\end{appendices}

\noindent
{\bf Acknowledgements.} The work of Bao was supported by Singapore MOE grant
MOE2019-T2-1-063 (R-146-000-296-112). The work of Zhao was supported by the Singapore MOE grant R-146-000-285-114.

\bibliographystyle{elsarticle-num}
\bibliography{thebib}

\end{document}